\documentclass[11pt]{amsart}

\scrollmode

\usepackage{times}
\usepackage{amsmath,graphicx}
\usepackage{latexsym}
\usepackage{amscd,amsthm,amssymb}
\usepackage[all]{xy}

\newtheorem{thm}{Theorem}[section]
\newtheorem{prop}[thm]{Proposition}
\newtheorem{lem}[thm]{Lemma}
\newtheorem{cor}[thm]{Corollary}

\theoremstyle{definition}
\newtheorem{ex}[thm]{Example}
\newtheorem{rem}[thm]{Remark}
\newtheorem{defn}[thm]{Definition}

\title[Generalized fibre sums and the canonical class]{Generalized fibre sums of 4-manifolds and the canonical class}
\author{M.~J.~D.~Hamilton}
\address{      Institut f\"ur Geometrie und Topologie\\
               Universit\"at Stuttgart\\
               Pfaffenwaldring 57\\
               70569 Stuttgart\\
               Germany}
\email{mark.hamilton@math.lmu.de}
\date{\today}

\begin{document}

\begin{abstract} In this paper we determine the integral homology and cohomology groups of a closed 4-manifold $X$ obtained as the generalized fibre sum of two closed 4-manifolds $M$ and $N$ along embedded surfaces of genus $g$ and self-intersection zero. If the homologies of the 4-manifolds are torsion free and the surfaces represent indivisible homology classes, we derive a formula for the intersection form of $X$.  If in addition the 4-manifolds $M$ and 
$N$ are symplectic and the surfaces symplectically embedded we derive a formula for the canonical class of the symplectic fibre sum.
\end{abstract}

\maketitle

\tableofcontents

\section{Introduction}
In this paper we are interested in the generalized fibre sum of closed oriented 4-manifolds $M$ and $N$ along closed embedded oriented surfaces $\Sigma_M$ and $\Sigma_N$ of genus $g$. We always assume that both surfaces represent non-torsion homology classes and have self-intersection zero, i.e.~their normal bundles are trivial, and choose embeddings
\begin{align*}
i_M\colon \Sigma&\rightarrow M\\
i_N\colon \Sigma&\rightarrow N
\end{align*}
that realize the surfaces as images of a fixed closed surface $\Sigma$ of genus $g$. The generalized fibre sum $X=X(\phi)=M\#_{\Sigma_M=\Sigma_N}N$ is defined as 
\begin{equation*}
X(\phi)=M'\cup_{\phi}N'
\end{equation*}
where $M'$ and $N'$ denote the manifolds with boundary $\Sigma\times S^1$ obtained by deleting the interior of tubular neighbourhoods $\Sigma\times D^2$ of the surfaces in $M$ and $N$ and $\phi$ is an orientation reversing diffeomorphism $\phi\colon \partial M'\rightarrow \partial N'$ that preserves the $S^1$ fibration and covers the diffeomorphism $i_N\circ i_M^{-1}$ between the surfaces. 

The generalized fibre sum, also called {\em Gompf sum}, is one of the most important techniques in constructing new 4-manifolds. It is therefore desirable to understand the basic topological invariants of a generalized fibre sum in the general situation, without having to do the calculations in each example anew. Specifically, we want to calculate the integral homology and cohomology groups and the intersection form of the 4-manifold $X$ obtained as a generalized fibre sum of $M$ and $N$ as above. To do so, we will first calculate in Section \ref{section hom cohom of M'} the homology and cohomology of the complement $M'=M\setminus \text{int}\,\nu\Sigma_M$ and determine the mapping induced by the gluing diffeomorphism $\phi$ on the homology of the boundaries $\partial M'$ and $\partial N'$ in Section \ref{subsect action phi}. The homology and cohomology groups of $X$ are then derived by Mayer-Vietoris arguments from the corresponding groups of $M'$ and $N'$.

For example, having calculated the first cohomology $H^1(X)$, we can determine the Betti numbers of $X$: Let $d$ denote the dimension of the kernel of the linear map of $\mathbb{R}$-vector spaces
\begin{equation*}
i_M\oplus i_N\colon H_1(\Sigma;\mathbb{R})\longrightarrow H_1(M;\mathbb{R})\oplus H_1(N;\mathbb{R}),
\end{equation*}
induced by the embeddings. Then we have:

\begin{cor} The Betti numbers of a generalized fibre sum $X=M\#_{\Sigma_M=\Sigma_N}N$ along surfaces $\Sigma_M$ and $\Sigma_N$ of genus $g$ and self-intersection zero are given by
\begin{align*}
b_0(X)&=b_4(X)=1\\
b_1(X)&=b_3(X)=b_1(M)+b_1(N)-2g+d\\
b_2(X)&=b_2(M)+b_2(N)-2+2d\\
b_2^+(X)&=b_2^+(M)+b_2^+(N)-1+d\\
b_2^-(X)&=b_2^-(M)+b_2^-(N)-1+d.
\end{align*} 
\end{cor}
We then derive formulae for the groups $H_1(X)$ and $H^2(X)$ with integer coefficients, see Theorems \ref{H 1 for X} and \ref{computation thm H2}. Similar formulae for the second cohomology can be found in several places in the literature, in particular in some special cases, for example \cite{DL, FSknot, FSnon, FSfam, FrM1, GS}. They are probably known to the experts, but we could not find a reference for them in the general case. Understanding the cohomology is also a necessary prerequisite for the formula of the canonical class that we prove in the second part of this paper.

In Section \ref{section calc intersection form of X} we derive a formula for the intersection form of $X$ in the case that the cohomologies of $M$, $N$ and $X$ are torsion free and the surfaces $\Sigma_M$ and $\Sigma_N$ represent indivisible classes. Let $B_M$ be a surface in $M$ such that $B_M\Sigma_M=1$ and let $P(M)$ denote the orthogonal complement to the subgroup $\mathbb{Z}B_M\oplus\mathbb{Z}\Sigma_M$ in the second cohomology of $M$. We then get a splitting
\begin{equation*}
H^2(M)=P(M)\oplus \mathbb{Z}B_M\oplus \mathbb{Z}\Sigma_M.
\end{equation*}
An analogous splitting exists for $H^2(N)$. We want to derive such a splitting for $H^2(X)$ together with a formula for the intersection form. First, a choice of framing for the surface $\Sigma_M$, i.e.~a trivialization of the normal bundle, determines a push-off $\Sigma^M$ of the surface into the boundary $\partial M'$ and this surface determines under inclusion a surface $\Sigma_X$ in $X$. A similar surface $\Sigma_X'$ is determined by a framing for $\Sigma_N$. Depending on the homology of $X$ and the gluing diffeomorphism, the surfaces $\Sigma_X$ and $\Sigma_X'$ do not necessarily define the same homology class in $X$. The surfaces $B_M$ and $B_N$ minus a disk sew together to define a surface $B_X$ in $X$ with intersection number $B_X\Sigma_X=1$. We then have:
\begin{thm} Let $X=M\#_{\Sigma_M=\Sigma_N}N$ be a generalized fibre sum of closed oriented 4-manifolds $M$ and $N$ along embedded surfaces $\Sigma_M,\Sigma_N$ of genus $g$ and self-intersection zero which represent indivisible homology classes. Suppose that the cohomology of $M$, $N$ and $X$ is torsion free. Then there exists a splitting
\begin{equation*}
H^2(X;\mathbb{Z})=P(M)\oplus P(N)\oplus (S'(X)\oplus R(X))\oplus (\mathbb{Z}B_X\oplus\mathbb{Z}\Sigma_X),
\end{equation*}
where 
\begin{equation*}
(S'(X)\oplus R(X))=(\mathbb{Z}S_1\oplus\mathbb{Z}R_1)\oplus\dotsc\oplus(\mathbb{Z}S_d\oplus\mathbb{Z}R_d).
\end{equation*}
The direct sums are all orthogonal, except the direct sums inside the brackets. In this decomposition of $H^2(X;\mathbb{Z})$, the restriction of the intersection form $Q_X$ to $P(M)$ and $P(N)$ is equal to the intersection form induced from $M$ and $N$ and has the structure
\begin{equation*}
\left(\begin{array}{cc} B_M^2+B_N^2 &1 \\ 1& 0 \\ \end{array}\right)
\end{equation*} 
on $\mathbb{Z}B_X\oplus\mathbb{Z}\Sigma_X$ and the structure   
\begin{equation*}
\left(\begin{array}{cc} S_i^2 &1 \\ 1& 0 \\ \end{array}\right)
\end{equation*}  
on each summand $\mathbb{Z}S_i\oplus\mathbb{Z}R_i$.
\end{thm}
In this formula the classes in $S'(X)$ are so-called {\em vanishing classes}, sewed together from surfaces in $M'$ and $N'$ which bound curves on the boundaries $\partial M'$ and $\partial N'$ that get identified under the gluing diffeomorphism $\phi$. The group $R(X)$ is the group of {\em rim tori} in $X$ and the integer $d$ is the dimension of the kernel of the map $i_M\oplus i_N$ above. The formula for the intersection form is similar to a formula for the simply-connected elliptic surfaces $E(n)$, which can be found, for example, in \cite{GS}. A natural question is how canonical this type of basis is. We will see that the choice of a basis for the rim tori in one fibre sum determines such a basis canonically in all fibre sums as the gluing diffeomorphism changes. In addition the basis for the rim tori determines the basis for the vanishing surfaces up to rim tori summands. The surfaces $\Sigma_X$ and $\Sigma_X'$ depend on the choice of framings and differ by a rim torus $R_C$ that depends on the gluing diffeomorphism. The rim tori themselves are independent of the choice of framing.

Given the decomposition of $H^2(M)$ as a direct sum above, we get an embedding of $H^2(M)$ into $H^2(X)$ by mapping $B_M$ to $B_X$, $\Sigma_M$ to $\Sigma_X$ and taking the identity on $P(M)$. This embedding does not in general preserve the intersection form, since $B_X^2=B_M^2+B_N^2$. There exists a similar embedding for $H^2(N)$ into $H^2(X)$, mapping $B_N$ to $B_X$ and $\Sigma_N$ to $\Sigma_X'$. Heuristically, we can think of $H^2(X)$ as being formed by joining $H^2(M)$ and $H^2(N)$ along their "nuclei" $\mathbb{Z}B_M\oplus\mathbb{Z}\Sigma_M$ and $\mathbb{Z}B_N\oplus\mathbb{Z}\Sigma_N$, which form the new nucleus $\mathbb{Z}B_X\oplus\mathbb{Z}\Sigma_X$ of $X$ by sewing together $B_M$ and $B_N$ to the surface $B_X$. There are additional summands coming from the vanishing classes and rim tori groups that do not exist in the closed manifolds $M$ and $N$ separately.

Finally, we consider the symplectic case: Suppose that $M$ and $N$ are symplectic manifolds and $\Sigma_M$ and $\Sigma_N$ symplectically embedded surfaces of genus $g$. Then the generalized fibre sum $X$ admits a symplectic form. In Section \ref{formula KX calculation} we derive a formula for the canonical class of $X$, which is minus the first Chern class of a compatible almost complex structure, under the same assumptions as in the theorem on the intersection form of $X$. In general, the canonical class has non-zero pairing with the vanishing surfaces, hence there is a rim tori contribution to the canonical class. The formula describes how the canonical class changes as we change the gluing diffeomorphism. It can be written as:

\begin{cor} Under the embeddings of $H^2(M)$ and $H^2(N)$ into $H^2(X)$, the canonical class of $X=M\#_{\Sigma_M=\Sigma_N}N$ is given by
\begin{equation*}
K_X=K_M+K_N+\Sigma_X+\Sigma_X'-(2g-2)B_X+\sum_{i=1}^dt_iR_i,
\end{equation*}
where $t_i= K_{X_0}S_i$.
\end{cor}
Here $X_0$ is the generalized fibre sum obtained from the trivial gluing diffeomorphism that identifies the push-offs. Note that the manifold $X_0$ and the surfaces $\Sigma_X$, $\Sigma_X'$ and $S_i$ depend on the choice of framings. However, we will see that the canonical class is independent of this choice and depends only on the abstract gluing diffeomorphism, as it should be. Using the formula for the canonical class and the intersection form, we can also directly check that
\begin{equation*}
K_X^2=K_M^2+K_N^2+8g-8.
\end{equation*}
This formula can be derived in an elementary way from the identity $K^2=2e+3\sigma$, where $e$ denotes the Euler characteristic and $\sigma$ the signature of the 4-manifold.

In particular, suppose that rim tori do not exist in $X$. Then we get 
\begin{equation*}
K_X=K_M+K_N+2\Sigma_X-(2g-2)B_X,
\end{equation*}
which generalizes the classical formula for the fibre sum along tori that can be found in the literature, e.g.~\cite{Smi}. In the final section we will consider a case in which there are rim tori but where the numbers $K_{X_0}S_i$ and hence the rim tori contribution can be calculated explicitly.

\subsection*{Acknowledgements} The content of this article is a revised version of a part of the author's Ph.D.~thesis. I am grateful to D.~Kotschick for supervising the thesis and to J.~Bowden for discussions about the canonical class. I would also like to thank the {\it Studienstiftung des deutschen Volkes} and the {\it Deutsche Forschungsgemeinschaft (DFG)} for financial support.

\section{Definition of the generalized fibre sum}\label{Sect Gompf sum general}

In the following, we use for a topological space $Y$ the abbreviations $H_*(Y)$ and $H^*(Y)$ to denote the homology and cohomology groups of $Y$ with $\mathbb{Z}$-coefficients. Other coefficients will be denoted explicitly. The homology class of an embedded, oriented surface and the surface itself are often denoted by the same symbol. Poincar\'e duality is often suppressed, so that a class and its Poincar\'e dual are denoted by the same symbol.

Let $M$ and $N$ be closed, oriented, connected 4-manifolds. Suppose that $\Sigma_M$ and $\Sigma_N$ are closed, oriented, connected embedded surfaces in $M$ and $N$ of the same genus $g$. Let $\nu\Sigma_M$ and $\nu\Sigma_N$ denote the normal bundles of $\Sigma_M$ and $\Sigma_N$. The normal bundle of the surface $\Sigma_M$ is trivial if and only if the self-intersection number $\Sigma_M^2$ is zero. This follows because the Euler class of the normal bundle is given by $e(\nu\Sigma_M)=i^*PD[\Sigma_M]$, where $i\colon  \Sigma_M\rightarrow M$ denotes the inclusion. Hence the evaluation of the Euler class on the fundamental class of $\Sigma_M$ is equal to the self-intersection number of $\Sigma_M$ in the 4-manifold $M$. From now on we will assume that $\Sigma_M$ and $\Sigma_N$ have zero self-intersection. 

For the construction of the generalized fibre sum we choose a closed oriented surface $\Sigma$ of genus $g$ and smooth embeddings
\begin{align*}
i_M\colon  \Sigma&\longrightarrow M \\
i_N\colon \Sigma&\longrightarrow N,
\end{align*}
with images $\Sigma_M$ and $\Sigma_N$. We assume that the orientation induced by the embeddings on $\Sigma_M$ and $\Sigma_N$ is the given one. 

Since the normal bundles of $\Sigma_M$ and $\Sigma_N$ are trivial, there exist trivial $D^2$-bundles $\nu\Sigma_M$ and $\nu\Sigma_N$ embedded in $M$ and $N$, forming closed tubular neighbourhoods for $\Sigma_M$ and $\Sigma_N$. We fix once and for all framings for the surfaces, i.e.~embeddings
\begin{align*}
\tau_M\colon \Sigma\times D^2&\longrightarrow M \\
\tau_N\colon \Sigma\times D^2&\longrightarrow N,
\end{align*} 
with images $\nu\Sigma_M$ and $\nu\Sigma_N$ which are fixed reference trivialisations for the normal bundles of the embedded surfaces $\Sigma_M$ and $\Sigma_N$. We denote the restriction of these maps to $\Sigma\times S^1$ by $\bar{\tau}_M$ and $\bar{\tau}_N$. They commute with the embeddings $i_M$ and $i_N$ above and the natural projections 
\begin{align*}
p\colon &\Sigma\times S^1\rightarrow \Sigma\\
p_M\colon &\partial\nu\Sigma_M\rightarrow \Sigma_M\\ 
p_N\colon &\partial\nu\Sigma_N\rightarrow \Sigma_N.
\end{align*}
Taking the image of $\Sigma$ times an arbitrary point on $S^1$, the framings $\bar{\tau}_M$ and $\bar{\tau}_N$ determine certain sections of the $S^1$-bundles $\partial\nu\Sigma_M$ and $\partial\nu\Sigma_N$. These sections correspond to {\em push-offs} of the surfaces $\Sigma_M$ and $\Sigma_N$ into the boundary of the tubular neighbourhoods. Since trivializations of vector bundles are linear, the framings are completely determined by such push-offs.

\begin{defn} We denote by $\Sigma^M$ and $\Sigma^N$ push-offs of $\Sigma_M$ and $\Sigma_N$ into the boundaries $\partial \nu \Sigma_M$ and $\partial \nu \Sigma_N$ given by the framings $\tau_M$ and $\tau_N$.
\end{defn}
We set 
\begin{align*}
M'&=M\setminus \text{int}\,\nu\Sigma_M\\
N'&=N\setminus \text{int}\,\nu\Sigma_N
\end{align*}
which are compact, oriented 4-manifolds with boundary. The orientations are chosen as follows: On $\Sigma\times D^2$ choose the orientation of $\Sigma$ followed by the standard orientation of $D^2$ given by $dx\wedge dy$. By choosing the orientations of $M$ and $N$ we can assume that the framings $\tau_M$ and $\tau_N$ induce orientation preserving embeddings of $\Sigma\times D^2$ into $M$ and $N$ as tubular neighbourhoods. We define the orientation of $\Sigma\times S^1$ to be the orientation of $\Sigma$ followed by the counter clockwise orientation of $S^1$. This determines orientations of $\partial M'$ and $\partial N'$. Both conventions together imply that the orientation of $\partial M'$ followed by the normal direction pointing out of $M'$ is the orientation of $M'$ induced from $M$. This is opposite to the standard orientation of $\partial M'$ (the normal direction pointing out of $M'$ followed by the standard orientation of $\partial M'$ is the orientation of $M'$). Similarly for $N$.

We want to glue $M'$ and $N'$ together using diffeomorphisms between the boundaries which preserve the fibration of the $S^1$-bundles $\partial\nu\Sigma_M$ and $\partial\nu\Sigma_N$ and cover the diffeomorphism $i_N\circ i_M^{-1}$. Since the group $\text{Diff}^+(S^1)$ retracts onto $SO(2)$, every orientation and fibre preserving diffeomorphism $\Sigma\times S^1\rightarrow \Sigma\times S^1$ covering the identity is isotopic to a diffeomorphism of the form
\begin{equation}\label{equation for C}
\begin{split}
F\colon \Sigma\times S^1&\rightarrow \Sigma\times S^1,\\
(z,\alpha)&\mapsto (z,C(z)\cdot\alpha),
\end{split}
\end{equation}
where $C\colon \Sigma\rightarrow S^1$ is a map and multiplication is in the group $S^1$. Conversely, every smooth map $C$ from $\Sigma$ to $S^1$ defines such an orientation and fibre  preserving diffeomorphism. Let $r$ denote the orientation reversing diffeomorphism
\begin{equation*}
r\colon \Sigma\times S^1\rightarrow \Sigma\times S^1, (z,\alpha)\mapsto (z,\overline{\alpha}),
\end{equation*}
where $S^1\subset\mathbb{C}$ is embedded in the standard way and $\overline{\alpha}$ denotes complex conjugation. Then the diffeomorphism
\begin{align*}
\rho=F\circ r\colon \Sigma\times S^1&\rightarrow \Sigma\times S^1,\\
(z,\alpha)&\mapsto (z,C(z)\overline{\alpha})
\end{align*}
is orientation reversing. We define 
\begin{equation}\label{rho phi}
\phi=\phi_C=\bar{\tau}_N\circ\rho\circ \bar{\tau}_M^{-1}.
\end{equation}
Then $\phi$ is an orientation reversing diffeomorphism $\phi\colon  \partial\nu\Sigma_M\rightarrow \partial\nu\Sigma_N$, preserving the circle fibres. If $C$ is a constant map then $\phi$ is a diffeomorphism which identifies push-offs of $\Sigma_M$ and $\Sigma_N$.

\begin{defn} Let $M$ and $N$ be closed, oriented, connected 4-manifolds $M$ and $N$ with embedded oriented surfaces $\Sigma_M$ and $\Sigma_N$ of genus $g$ and self-intersection $0$. The {\em generalized fibre sum} of $M$ and $N$ along $\Sigma_M$ and $\Sigma_N$, determined by the diffeomorphism $\phi$, is given by
\begin{equation*}
X(\phi)=M'\cup_{\phi} N'.
\end{equation*}
$X(\phi)$ is again a differentiable, closed, oriented, connected 4-manifold.
\end{defn}
See the references \cite{Go} and \cite{McW} for the original construction. The generalized fibre sum is often denoted by $M\#_{\Sigma_M=\Sigma_N}N$ or $M\#_\Sigma N$ and is also called the Gompf sum or the normal connected sum. 

The differentiable structure on $X$ is defined in the following way: We identify the interior of slightly larger tubular neighbourhoods $\nu\Sigma_M'$ and $\nu\Sigma_N'$ via the framings $\tau_M$ and $\tau_N$ with $\Sigma\times D$ where $D$ is an open disk of radius $1$. We think of $\partial M'$ and $\partial N'$ as the 3-manifold $\Sigma\times S$, where $S$ denotes the circle of radius $\scriptstyle{\frac{1}{\sqrt{2}}}$. Hence the tubular neighbourhoods $\nu\Sigma_M$ and $\nu\Sigma_N$ above have in this convention radius $\scriptstyle{\frac{1}{\sqrt{2}}}$. We also choose polar coordinates $r,\theta$ on $D$. The manifolds $M\setminus\Sigma_M$ and $N\setminus \Sigma_N$ are glued together along $\text{int}\,\nu\Sigma_M'\setminus \Sigma_M$ and $\text{int}\,\nu\Sigma_N'\setminus \Sigma_N$ by the diffeomorphism
\begin{equation}\label{eq gluing diff Phi}
\begin{split}
\Phi\colon \Sigma\times (D\setminus \{0\})&\rightarrow \Sigma\times (D\setminus\{0\}) \\
(z, r,\theta)&\mapsto (z, \sqrt{1-r^2}, C(z)-\theta).
\end{split}
\end{equation}
This diffeomorphism is orientation {\em preserving} because it reverses on the disk the orientation on the boundary circle and the inside-outside direction. It preserves the fibration by punctured disks and identifies $\partial M'$ and $\partial N'$ via $\phi$. 

\begin{defn} Let $\Sigma_X$ denote the genus $g$ surface in $X$ given by the image of the push-off $\Sigma^M$ under the inclusion $M'\rightarrow X$. Similarly, let $\Sigma_X'$ denote the genus $g$ surface in $X$ given by the image of the push-off $\Sigma^N$ under the inclusion $N'\rightarrow X$.
\end{defn}

In general (depending on the diffeomorphism $\phi$ and the homology of $X$) the surfaces $\Sigma_X$ and $\Sigma_X'$ do not represent the same homology class in $X$ but differ by a rim torus, cf.~Lemma \ref{lem difference sigmaXX' RC}.

\subsection{Isotopic gluing diffeomorphisms}
Different choices of gluing diffeomorphisms $\phi$ can result in non-diffeomorphic manifolds $X(\phi)$. However, if $\phi$ and $\phi'$ are isotopic, then $X(\phi)$ and $X(\phi')$ are diffeomorphic. We want to determine how many different isotopy classes of gluing diffeomorphisms $\phi$ of the form above exist (note that the manifold $X$ of course also depends on the choice of the embeddings $i_M$ and $i_N$ realizing the surfaces in $M$ and $N$ as the image of a fixed surface $\Sigma$). We first make the following definition:
\begin{defn} Let $C\colon \Sigma\rightarrow S^1$ be the map used to define the gluing diffeomorphism $\phi$ in equation \eqref{equation for C}. Then the integral cohomology class $[C]\in H^1(\Sigma)$ is defined by pulling back the standard generator of $H^1(S^1)$. We sometimes denote $[C]$ by $C$ if a confusion is not possible.
\end{defn}
Suppose that
\begin{equation*}
C, C'\colon \Sigma\rightarrow S^1,
\end{equation*}
are smooth maps which determine self-diffeomorphisms $\rho$ and $\rho'$ of $\Sigma\times S^1$ and gluing diffeomorphisms $\phi, \phi'\colon \partial \nu\Sigma_M\rightarrow \partial\nu\Sigma_N$ as before. 
\begin{prop}\label{prop [C] determines phi} The diffeomorphisms $\phi,\phi'\colon  \partial\nu\Sigma_M\longrightarrow \partial\nu\Sigma_N$ are smoothly isotopic if and only if $[C]=[C']\in H^1(\Sigma)$. In particular, if the maps $C$ and $C'$ determine the same cohomology class, then the generalized fibre sums $X(\phi)$ and $X(\phi')$ are diffeomorphic.
\end{prop} 
\begin{proof} Suppose that $\phi$ and $\phi'$ are isotopic. The equation
\begin{equation*}
\rho=\bar{\tau}^{-1}_N\circ\phi\circ\bar{\tau}_M,
\end{equation*}
implies that the diffeomorphisms $\rho,\rho'$ are also isotopic, hence homotopic. The maps $C,C'$ can be written as
\begin{equation*}
C=pr\circ\rho\circ \iota,\quad C'=pr\circ\rho'\circ\iota,
\end{equation*}
where $\iota\colon \Sigma\rightarrow\Sigma\times S^1$ denotes the inclusion $x\mapsto (x,1)$ and $pr$ denotes the projection onto the second factor in $\Sigma\times S^1$. This implies that $C$ and $C'$ are homotopic, hence the cohomology classes $[C]$ and $[C']$ coincide. 

Conversely, if the cohomology classes $[C]$ and $[C']$ coincide, then $C$ and $C'$ are homotopic maps. We can choose a smooth homotopy
\begin{align*}
\Delta\colon \Sigma\times [0,1]&\longrightarrow S^1,\\
(x,t)&\mapsto \Delta(x,t)
\end{align*}
with $\Delta_0=C$ and $\Delta_1=C'$. Define the map
\begin{align*}
R\colon (\Sigma\times S^1)\times [0,1]&\longrightarrow \Sigma\times S^1, \\
(x,\alpha,t)&\mapsto  R_t(x,\alpha),
\end{align*}
where 
\begin{equation*}
R_t(x,\alpha)=(x,\Delta(x,t)\cdot \overline{\alpha}).
\end{equation*}
Then $R$ is a homotopy between $\rho$ and $\rho'$. The maps $R_t\colon \Sigma\times S^1\rightarrow \Sigma\times S^1$ are diffeomorphisms with inverse
\begin{equation*}
(y,\beta)\mapsto (y,\overline{\Delta(y,t)^{-1}\cdot\beta}),
\end{equation*}
where $\Delta(y,t)^{-1}$ denotes the inverse as a group element in $S^1$. Hence $R$ is an isotopy between $\rho$ and $\rho'$ that defines via the framings $\tau_M$ and $\tau_N$ an isotopy between $\phi$ and $\phi'$.
\end{proof}

\subsection{Action of the gluing diffeomorphism on homology}\label{subsect action phi}

In this subsection we determine the action of the gluing diffeomorphism $\phi\colon \partial M'\rightarrow \partial N'$ on the homology of the boundaries $\partial M'$ and $\partial N'$. We first choose bases for these homology groups using the framings $\tau_M$ and $\tau_N$.

First choose a basis for $H_1(\Sigma)$ consisting of oriented embedded loops $\gamma_1,\dotsc,\gamma_{2g}$ in $\Sigma$. For each index $i$, we denote the loop $\gamma_i\times\{*\}$ in $\Sigma\times S^1$ also by $\gamma_i$. Let $\sigma$ denote the loop $\{*\}\times S^1$ in $\Sigma\times S^1$. Then the loops
\begin{equation*}
\gamma_1,\dotsc,\gamma_{2g},\sigma
\end{equation*}
represent homology classes (denoted by the same symbols) which determine a basis for $H_1(\Sigma\times S^1)\cong\mathbb{Z}^{2g+1}$. 
\begin{defn} The basis for the first homology of $\partial\nu\Sigma_M$ and $\partial\nu\Sigma_N$ is chosen as:
\begin{align*}
\gamma_i^M={\bar{\tau}_M}{_*}\gamma_i,&\quad \sigma^M={\bar{\tau}_M}{_*}\sigma\\
\gamma_i^N={\bar{\tau}_N}{_*}\gamma_i,&\quad \sigma^N={\bar{\tau}_N}{_*}\sigma.
\end{align*}
\end{defn}

For $H_2(\Sigma\times S^1)$ we define a basis consisting of elements $\Gamma_1,\dotsc,\Gamma_{2g},\Sigma$ such that the intersection numbers satisfy the following equations:
\begin{align*}
\Gamma_i\cdot\gamma_i&=1,\quad i=1,\dotsc,2g,\\
\Sigma\cdot\sigma&=1,
\end{align*}
and all other intersection numbers are zero. By Poincar\'e duality, this basis can be realized by choosing the dual basis $\gamma_1^*,\dotsc,\gamma_{2g}^*,\sigma^*$ in the group  
\begin{equation*}
H^1(\Sigma\times S^1)=\text{Hom}(H_1(\Sigma\times S^1),\mathbb{Z})
\end{equation*}
to the basis for the first homology above and defining
\begin{align*}
\Gamma_i&=PD(\gamma_i^*),\quad i=1,\dotsc,2g,\\
\Sigma&=PD(\sigma^*).
\end{align*}
\begin{defn} The basis for the second homology of $\partial\nu\Sigma_M$ and $\partial\nu\Sigma_N$ is defined as:
\begin{align*}
\Gamma_i^M={\bar{\tau}_M}{_*}\Gamma_i,&\quad \Sigma^M={\bar{\tau}_M}{_*}\Sigma\\
\Gamma_i^N={\bar{\tau}_N}{_*}\Gamma_i,&\quad \Sigma^N={\bar{\tau}_N}{_*}\Sigma.
\end{align*}
\end{defn}
The class $\Sigma^M$ can be represented by a push-off of $\Sigma_M$ determined by the framing $\tau_M$. The classes $\Gamma_i^M$ have the following interpretation: We can choose a basis $\pi_1,\dotsc,\pi_{2g}$ of $H_1(\Sigma)$ such that
\begin{equation*}
\pi_i\cdot\gamma_j=-\delta_{ij}
\end{equation*}
for all indices $i,j$. The existence of such a basis follows because the intersection form on $H_1(\Sigma)$ is symplectic. The intersection numbers of the immersed tori $\pi_i\times \sigma$ in $\Sigma\times S^1$ with the curves $\gamma_j$ are given by
\begin{equation*}
(\pi_i\times\sigma)\cdot\gamma_j=\delta_{ij},
\end{equation*}
hence these tori represent the classes $\Gamma_i$.

\begin{defn}\label{defn coeff ai} For the basis $\gamma_1,\dotsc,\gamma_{2g}$ of $H_1(\Sigma)$ above define the integers
\begin{align*}
a_i&=\text{deg}(C\circ \gamma_i\colon S^1\rightarrow S^1) \\
                     &=\langle [C],\gamma_i\rangle =\langle C,\gamma_i\rangle \in\mathbb{Z}.
\end{align*}
The integers $a_i$ together determine the cohomology class $[C]$. Since the map $C$ can be chosen arbitrarily, the integers $a_i$ can (independently) take any possible value.
\end{defn}

\begin{lem}\label{act phi hom 1} The map $\phi_*\colon H_1(\partial\nu\Sigma_M)\rightarrow H_1(\partial\nu\Sigma_N)$ is given by
\begin{align*}
\phi_*\gamma_i^M&= \gamma_i^N+a_i\sigma^N,\quad i=1,\dotsc,2g\\
\phi_*\sigma^M&= -\sigma^N.
\end{align*}
\end{lem}
\begin{proof} We have
\begin{equation*}
\rho(\gamma_i(t),*)=(\gamma_i(t),(C\circ \gamma_i)(t)\cdot\overline{*}),
\end{equation*}
which implies on differentiation $\rho_*\gamma_i=\gamma_i+a_i\sigma$ for all $i=1,\dotsc,2g$. Similarly, 
\begin{equation*}
\rho (*,t)=(*,C(*)\cdot\overline{t}),
\end{equation*}
which implies  $\rho_*\sigma=-\sigma$. The claim follows from these equations and equation \eqref{rho phi}. 
\end{proof} 

\begin{lem}\label{computation action phi H_2}
The map $\phi_*\colon H_2(\partial\nu\Sigma_M)\rightarrow H_2(\partial\nu\Sigma_N)$ is given by
\begin{align*}
\phi_*\Gamma_i^M&=-\Gamma_i^N,\quad i=1,\dotsc,2g\\
\phi_*\Sigma^M&=-\left(\sum_{i=1}^{2g}a_i\Gamma_i^N\right)+\Sigma^N.
\end{align*}
\end{lem}
\begin{proof} We first compute the action of $\rho$ on the first cohomology of $\Sigma\times S^1$. By the proof of Lemma \ref{act phi hom 1}, 
\begin{align*}
(\rho^{-1})_*\gamma_i&= \gamma_i+a_i\sigma,\quad i=1,\dotsc,2g\\
(\rho^{-1})_*\sigma&=-\sigma.
\end{align*}
We claim that
\begin{align*}
(\rho^{-1})^*(\gamma_i^*)&= \gamma_i^*,\quad i=1,\dotsc,2g,\\
(\rho^{-1})^*(\sigma^*)&= \left(\sum_{i=1}^{2g}a_i\gamma_i^*\right)-\sigma^*.
\end{align*}
This is easy to check by evaluating both sides on the given basis of $H_1(\Sigma\times S^1)$ and using $\langle (\rho^{-1})^*\mu, v\rangle = \langle \mu, (\rho^{-1})_*v\rangle$. By the formula
\begin{equation}\label{cupcap}
\lambda_*(\lambda^*\alpha\cap\beta)=\alpha\cap\lambda_*\beta,
\end{equation}
for continuous maps $\lambda$ between topological spaces, homology classes $\beta$ and cohomology classes $\alpha$ (see \cite[Chapter VI, Theorem 5.2]{Br}), we get for all classes $\mu\in H^*(\Sigma\times S^1)$,
\begin{equation}\label{or rev cup cap}
\begin{split}
\rho_*PD(\rho^*\mu)&=\rho_*(\rho^*\mu\cap [\Sigma\times S^1])\\
&=\mu\cap \rho_*[\Sigma\times S^1] \\
&=-\mu\cap [\Sigma\times S^1]\\
&=-PD(\mu).  
\end{split}
\end{equation}
since $\rho$ is orientation reversing. This implies $\rho_*PD(\mu)=-PD((\rho^{-1})^*\mu)$ and hence
\begin{align*}
\rho_*\Gamma_i&=-\Gamma_i,\quad i=1,\dotsc,2g,\\
\rho_*\Sigma&=-\left(\sum_{i=1}^{2g}a_i\Gamma_i\right)+\Sigma.
\end{align*}
The claim follows from this.
\end{proof}
\begin{prop}\label{diffeom phi determined by C} The diffeomorphism $\phi$ is determined up to isotopy by the difference of the homology classes $\phi_*\Sigma^M$ and $\Sigma^N$ in $\partial \nu \Sigma_N$. 
\end{prop}
\begin{proof} This follows because by the formula in Lemma \ref{computation action phi H_2} above, the difference determines the coefficients $a_i$. Hence it determines the class $[C]$ and by Proposition \ref{prop [C] determines phi} the diffeomorphism $\phi$ up to isotopy.
\end{proof}
\begin{lem}\label{computation action phi H^2} The map $\phi^*\colon H^2(\partial\nu\Sigma_N)\rightarrow H^2(\partial\nu\Sigma_M)$ is given by
\begin{align*}
\phi^*PD(\gamma_i^N)&=-PD(\gamma_i^M)-a_iPD(\sigma^M),\quad i=1,\dotsc,2g\\
\phi^*PD(\sigma^N)&=PD(\sigma^M).
\end{align*}
\end{lem}
\begin{proof}
From equation \eqref{or rev cup cap} we have
\begin{equation*}
\rho^*PD(\alpha)=-PD((\rho^{-1})_*\alpha)
\end{equation*}
for $\alpha\in H_1(\Sigma\times S^1)$. Together with the first two equations in the proof of Lemma \ref{computation action phi H_2} this implies the claim. 
\end{proof}
We fix the following notation for some inclusions:
\begin{align*}
\rho_M\colon &M'\rightarrow M\\
\mu_M\colon &\partial\nu \Sigma_M\rightarrow M'\\
j_M\colon &\sigma^M \rightarrow M'\\
\eta_M\colon &M'\rightarrow X,
\end{align*}
and corresponding maps for $N$. For simplicity, the maps induced on homology and homotopy groups will often be denoted by the same symbol.

\section{The homology and cohomology of $M'$}\label{section hom cohom of M'}

Let $M$ be a closed, oriented $4$-manifold and $\Sigma_M\subset M$ a closed, oriented, connected embedded surface of genus $g$ and self-intersection zero. We denote a closed tubular neighbourhood of $\Sigma_M$ by $\nu \Sigma_M$ and let $M'$ denote the complement
\begin{equation*}
M'=M\setminus \text{int}\,\nu \Sigma_M.
\end{equation*}
Then $M'$ is an oriented manifold with boundary $\partial\nu \Sigma_M$. As above we choose a fixed closed oriented surface $\Sigma$ of genus $g$ and an embedding 
\begin{equation*}
i\colon \Sigma\rightarrow M
\end{equation*}
with image $\Sigma_M$. We continue to use the same notations for the framing and the embeddings as in the previous section. However, we will often drop in this section the index $M$ on the maps to keep the formulae notationally more simple.

We always assume in this section that the surface $\Sigma_M$ represents a non-torsion class, denoted by the same symbol $\Sigma_M\in  H_2(M)$. On the closed 4-manifold $M$, the Poincar\'e dual of $\Sigma_M$ acts as a homomorphism on $H_2(M)$, 
\begin{equation*}
\langle PD(\Sigma_M),-\rangle\colon H_2(M)\longrightarrow \mathbb{Z}.
\end{equation*}
Since the homology class $\Sigma_M$ is non-torsion, the image of this homomorphism is non-zero and hence a subgroup of $\mathbb{Z}$ of the form $k_M\mathbb{Z}$ with $k_M>0$. We assume that $\Sigma_M$ is divisible by $k_M$, i.e.~there exists a class $A_M\in H_2(M)$ such that $\Sigma_M=k_MA_M$. This is always true, for example, if $H_2(M)\cong H^2(M)$ is torsion free. The class $A_M$ is primitive in the sense that the image of $PD(A_M)$ on $H_2(M)$ is all of $\mathbb{Z}$.

In the following calculations we will often use the Mayer-Vietoris sequence for the decomposition $M=M'\cup \nu\Sigma_M$:
\begin{equation*}
\ldots\rightarrow H_k(\partial M')\rightarrow H_k(M')\oplus H_k(\Sigma)\rightarrow H_k(M)\rightarrow H_{k-1}(\partial M')\rightarrow\ldots 
\end{equation*}
with homomorphisms
\begin{align*}
 H_k(\partial M')\rightarrow H_k(M')\oplus H_k(\Sigma),&\quad \alpha\mapsto (\mu_*\alpha,p_*\alpha) \\
H_k(M')\oplus H_k(\Sigma)\rightarrow H_k(M),&\quad (x,y)\mapsto \rho_* x-i_* y.
\end{align*} 
We also use the Mayer-Vietoris sequence for cohomology groups.

\begin{lem}\label{eq poinc dual H2M'dM'} The following diagram commutes up to sign for every integer $p$:
\begin{equation*}
\begin{CD}
H_p(M',\partial M') @>\partial >> H_{p-1}(\partial M')\\
@V\cong VV @V\cong VV \\
H^{4-p}(M') @>\mu^*>> H^{4-p}(\partial M')
\end{CD}
\end{equation*}
where the vertical isomorphisms are Poincar\'e duality.
\end{lem}
A proof for this lemma can be found in \cite[Chapter VI, Theorem 9.2]{Br}. We also need the following lemma:

\begin{lem}\label{lem poinc dual HMM' embedding} There is a commuting diagram for every integer $m$:
\begin{equation*}
\begin{CD}
H^m(M,M') @> >> H^m(M)\\
@V\cong VV @V\cong VV \\
H_{4-m}(\Sigma_M) @>i_*>> H_{4-m}(M)
\end{CD}
\end{equation*}
where the upper horizontal homomorphism comes from the long exact sequence in cohomology associated to the pair $(M,M')$ and the vertical maps are isomorphisms.
\end{lem}
A proof for this lemma follows from a version of Poincar\'e-Lefschetz duality as in \cite[Chapter VI, Corollary 8.4]{Br}. The isomorphism on the right is Poincar\'e duality and the isomorphism on the left can be thought of as
\begin{equation*}
H^m(M,M')\cong H^m(\nu \Sigma_M,\partial\nu\Sigma_M)\cong H_{4-m}(\nu\Sigma_M)\cong H_{4-m}(\Sigma_M)
\end{equation*}
induced by excision, Poincar\'e duality and the deformation retraction $\nu\Sigma_M\rightarrow \Sigma_M$.  

\subsection{Calculation of $H^2(M')$}
We begin with the calculation of the second cohomology of the complement $M'$.
\begin{prop}\label{lem H^2 M, M'} There exists a short exact sequence
\begin{equation*}
0\longrightarrow H^2(M)/\mathbb{Z}\Sigma_M\stackrel{\rho^*}{\longrightarrow}H^2(M')\longrightarrow \text{ker}(i\colon H_1(\Sigma_M)\rightarrow H_1(M))\longrightarrow 0.
\end{equation*}
This sequence splits, because $H_1(\Sigma_M)$ is torsion free. Hence there exists an isomorphism
\begin{equation}\label{splitting of H2M' as direct sum}
H^2(M')\cong \left(H^2(M)/\mathbb{Z}\Sigma_M\right)\oplus \text{ker}\,i.
\end{equation}
\end{prop}

\begin{proof} We consider the following part of the long exact sequence in cohomology associated to the pair $(M,M')$:
\begin{equation*}
\ldots\rightarrow H^2(M,M')\rightarrow H^2(M)\stackrel{\rho^*}{\rightarrow}H^2(M')\stackrel{\partial}{\rightarrow}H^3(M,M')\rightarrow H^3(M)\rightarrow \ldots
\end{equation*}
The claim follows by applying Lemma \ref{lem poinc dual HMM' embedding} for $m=2,3$.
\end{proof}

\begin{rem}\label{rem explicit splitting of H2M' as direct sum} An explicit splitting can be defined as follows: The images of loops representing a basis of $\text{ker}\,i$ under the embedding $i\colon \Sigma\rightarrow M$ bound surfaces in $M$. We can lift these curves to the boundary of $M'$ such that they bound in $M'$, see Lemma \ref{parallel curve adapted}. In this way the elements in $\text{ker}\,i$ determine classes in $H^2(M')\cong H_2(M',\partial M')$. 
\end{rem}

\begin{defn} Choose a class $B_M\in H^2(M)$ with $B_M\cdot A_M=1$. Such a class exists because $A_M$ is indivisible. We denote the image of this class in $H^2(M')\cong H_2(M',\partial M')$ by $B_M'$.
\end{defn}

The surface representing $B_M$ can be chosen such that it intersects the surface $\Sigma_M$ in precisely $k_M$ transverse positive points by starting with any surface that has intersection $k_M$ with $\Sigma_M$ and then possibly increasing the genus. Hence there also exists a surface $B_M'$ in $M'$ bounding $k_M$ parallel copies of the meridian $\sigma^M$.

Consider the subgroup in $H^2(M)$ generated by the classes $B_M$ and $A_M$. (This subgroup corresponds to the {\em Gompf nucleus} in elliptic surfaces, defined as a regular neighbourhood of a cusp fibre and a section \cite{Gnuc, GS}.)
\begin{defn} Let $P(M)=(\mathbb{Z}B_M\oplus \mathbb{Z}A_M)^\perp$ denote the orthogonal complement in $H^2(M)$ with respect to the intersection form. We call the elements in $P(M)$ {\em perpendicular classes}.
\end{defn}
Since $A_M^2=0$ and $A_M\cdot B_M=1$, the intersection form on these (indivisible) elements looks like
\begin{equation*}
\left(\begin{array}{cc} B_M^2 &1 \\ 1& 0 \\ \end{array}\right).
\end{equation*}
Since this form is unimodular it follows that there exists a direct sum decomposition
\begin{equation}\label{decomp of H^2 in A,B,P}
H^2(M)=\mathbb{Z}B_M\oplus \mathbb{Z}A_M\oplus P(M).
\end{equation}
The restriction of the intersection form to $P(M)$ modulo torsion is again unimodular (see \cite[Lemma 1.2.12]{GS}) and $P(M)$ has rank equal to $b_2(M)-2$. The following lemma determines the decomposition of an arbitrary element in $H^2(M)$ under the direct sum \eqref{decomp of H^2 in A,B,P}:

\begin{lem}\label{decomp alpha P(M)} For every element $\alpha\in H^2(M)$ define a class $\overline{\alpha}$ by the equation 
\begin{equation}\label{orth proj P(M)}
\alpha=(\alpha\cdot A_M)B_M+(\alpha\cdot B_M-B_M^2(\alpha \cdot A_M))A_M +\overline{\alpha}.
\end{equation}
Then $\overline{\alpha}$ is the component of $\alpha$ in the subgroup $P(M)$. 
\end{lem}
\begin{proof}
Writing $\alpha=aA_M+bB_M+\overline{\alpha}$, the intersections of $\overline{\alpha}$ with $A_M$ and $B_M$ have to vanish. This determines the coefficients $a$ and $b$. 
\end{proof}
\begin{defn} We define $P(M)_{A_M}=H^2(M)/\left(\mathbb{Z}\Sigma_M\oplus\mathbb{Z}B_M\right)=\mathbb{Z}_{k_M}A_M\oplus P(M)$.
\end{defn}
\begin{prop} There exists an isomorphism
\begin{equation*}
H^2(M')\cong P(M)_{A_M}\oplus \mathbb{Z}B_M\oplus \text{ker}\,i.
\end{equation*}
\end{prop} 
This proposition shows that the cohomology group $H^2(M')$ decomposes into elements in the interior of $M'$ coming from $M$ (i.e. elements in $P(M)_{A_M}$), classes bounding multiples of the meridian to $\Sigma_M$ (multiples of $B_M$) and classes bounding curves on the boundary $\partial M'$ which are lifts of curves on $\Sigma_M$ (elements of $\text{ker}\,i$).

\subsection{Calculation of $H_1(M')$ and $H^1(M')$}\label{subsec calculation H1M' H^1M'}

In this subsection we calculate the first homology and cohomology of the complement $M'$.
\begin{prop}\label{calculation cohomology H^1M'} The map $\rho$ induces an isomorphism $H^1(M')\cong H^1(M)$.
\end{prop}
\begin{proof}
We consider the following part of the long exact sequence in cohomology associated to the pair $(M,M')$:
\begin{equation*}
0\rightarrow H^1(M,M')\rightarrow H^1(M)\stackrel{\rho^*}{\rightarrow}H^1(M')\stackrel{\partial}{\rightarrow}H^2(M,M')\rightarrow H^2(M)\rightarrow \ldots
\end{equation*}
By Lemma \ref{lem poinc dual HMM' embedding} it follows that $H^1(M,M')=0$ and the map $H^2(M,M')\rightarrow H^2(M)$ is under Poincar\'e duality equivalent to the map
\begin{align*}
i\colon H_2(\Sigma_M)\cong\,&\mathbb{Z}\rightarrow H_2(M)\\
&a\mapsto a\Sigma_M.
\end{align*}
This map is injective, since the class $\Sigma_M$ is non-torsion. Hence by exactness the map $\rho^*\colon H^1(M)\rightarrow H^1(M')$ is an isomorphism. 
\end{proof}
We can now calculate the first homology of the complement.
\begin{prop}\label{H_1 complement} There exists an isomorphism $H_1(M')\cong H_1(M)\oplus \mathbb{Z}_{k_M}$. 
\end{prop} 
\begin{proof}By Proposition \ref{lem H^2 M, M'} we have 
\begin{align*}
\text{Tor}H^2(M')&\cong \text{Tor}(H^2(M)/\mathbb{Z}PD(\Sigma_M))\\
& \cong \text{Tor}(H^2(M)/\mathbb{Z}k_MPD(A_M))\\
&\cong \text{Tor}H^2(M)\oplus \mathbb{Z}PD(A_M)/k_M\mathbb{Z}PD(A_M)\\
&\cong \text{Tor}H^2(M)\oplus \mathbb{Z}_{k_M}PD(A_M).
\end{align*}
The third step follows because the class $A_M$ is primitive and of infinite order. The Universal Coefficient Theorem implies that 
\begin{equation*}
\text{Tor}H^2(M)=\text{Ext}(H_1(M),\mathbb{Z})\cong \text{Tor}H_1(M),
\end{equation*}
and similarly for $M'$. This implies
\begin{equation*}
\text{Tor}H_1(M')\cong \text{Tor}H_1(M)\oplus \mathbb{Z}_{k_M}.
\end{equation*}
Using again the Universal Coefficient Theorem we get
\begin{align*}
H_1(M')&\cong H^1(M')\oplus \text{Tor}H_1(M')\\
&\cong H^1(M)\oplus \text{Tor}H_1(M)\oplus \mathbb{Z}_{k_M}\\
&\cong H_1(M)\oplus \mathbb{Z}_{k_M}.
\end{align*}
\end{proof}
A similar calculation has been done in \cite{HsSz} and \cite{Ro} for the case of a 4-manifold $M$ under the assumption $H_1(M)=0$. The fundamental group of the complement $M'$ can be calculated as follows:
\begin{prop}\label{fund M and M'} The fundamental groups of $M$ and $M'$ are related by
\begin{equation*} \pi_1(M)\cong \pi_1(M')/N(\sigma^M),
\end{equation*}
where $N(\sigma^M)$ denotes the normal subgroup in $\pi_1(M')$ generated by the meridian $\sigma^M$ to the surface $\Sigma_M$.
\end{prop} 
The proof, which we omitt, is an application of the Seifert-van Kampen theorem (for a proof see for example \cite[Appendix]{MHthesis}). Taking the abelianization of the exact sequence
\begin{equation*}
1\rightarrow N(\sigma^M)\rightarrow \pi_1(M')\rightarrow \pi_1(M)\rightarrow 1
\end{equation*}
we get with Proposition \ref{H_1 complement}:
\begin{cor}\label{1 hom M M'} The first integral homology groups of $M'$ and $M$ are related by the exact sequence 
\begin{equation}\label{eq H1 M M'}
0\rightarrow \mathbb{Z}_{k_M}\stackrel{j}{\rightarrow}H_1(M')\stackrel{\rho}{\rightarrow}H_1(M)\rightarrow 0
\end{equation}
which splits. The image of $j$ is generated by the meridian $\sigma^M$ to the surface $\Sigma_M$.  
\end{cor}
{\em A priori} we only know that the kernel of $\rho$ is a finite cyclic subgroup isomorphic to $\mathbb{Z}_n$ generated by $\sigma^M$. We then use the following lemma:
\begin{lem}
Let $A$ be a finitely generated abelian group and $A'=A\oplus \mathbb{Z}_m$. If $A'$ contains a subgroup $H$ isomorphic to $\mathbb{Z}_n$ and $A'/H\cong A$, then $n=m$.
\end{lem}
The proof follows because $H$ is contained in the torsion subgroup of $A'$. We then count the elements of this torsion subgroup modulo $H$. 

Before we continue with the calculation of $H_2(M')$, we derive explicit expressions for the maps
\begin{equation*}
\mu^*\colon H^k(M')\rightarrow H^k(\partial M'),\quad k=1,2
\end{equation*}
and choose certain framings such that $\mu_*\colon H_1(\partial M')\rightarrow H_1(M')$ has a normal form.

\subsection{Adapted framings}\label{subsect adapted framings}

In this subsection, we define a particular class of framings $\tau_M$ which are adapted to the splitting of $H_1(M')$ into $H_1(M)$ and the torsion group determined by the meridian of $\Sigma_M$ as in Proposition \ref{H_1 complement}. This is a slightly "technical" issue which will make the calculations much easier. If the homology class representing $\Sigma_M$ is indivisible ($k_M=1$), then there is no restriction and every framing is adapted. 

By Corollary \ref{1 hom M M'} there exists an isomorphism of the form
\begin{align*}
s\colon H_1(M')&\longrightarrow H_1(M)\oplus\mathbb{Z}_{k_M}\\
\alpha&\mapsto (\rho_*\alpha,\mathcal{A}(\alpha)).
\end{align*}
The framing $\tau_M$ should be compatible with this isomorphism in the following way: The composition
\begin{equation}\label{comp muM smathcalA}
H_1(\partial M')\stackrel{\mu}{\rightarrow} H_1(M')\stackrel{s}{\rightarrow} H_1(M)\oplus \mathbb{Z}_{k_M}
\end{equation}
should be given on generators by
\begin{align*}
{\gamma_i^M}&\mapsto(i_*\gamma_i,0)\\
\sigma^M&\mapsto (0,1).
\end{align*}
Consider the exact sequence
\begin{equation}\label{eq Mayer Vietoris for adapted framings}
H_1(\partial M')\rightarrow H_1(M')\oplus H_1(\Sigma)\rightarrow H_1(M),
\end{equation}
coming from the Mayer-Vietoris sequence for $M$. It maps
\begin{align*}
\gamma_i^M&\mapsto (\mu_*\gamma_i^M,\gamma_i)\mapsto \rho_*\mu_*\gamma_i^M-i_*\gamma_i\\
\sigma^M&\mapsto (\mu_*\sigma^M,0)\,\,\mapsto \rho_*\mu_*\sigma^M.
\end{align*}
By exactness of the Mayer-Vietoris sequence, we have $\rho_*\mu_*\gamma_i^M=i_*\gamma_i$ and $\rho_*\mu_*\sigma^M=0$, where as before $\gamma_i^M$ is determined by $\gamma_i$ via the trivialization $\tau_M$. The isomorphism $s$ above maps
\begin{align*}
\mu_*\gamma_i^M&\mapsto (\rho_*\mu_*\gamma_i^M,\mathcal{A}(\mu_*\gamma_i^M))=(i_*\gamma_i,\mathcal{A}(\mu_*\gamma_i^M))\\
\mu_*\sigma^M&\mapsto (0,1).
\end{align*} 
Let $[c^M_i]$ denote the numbers $\mathcal{A}(\mu_*\gamma_i^M)\in \mathbb{Z}_{k_M}$. It follows that the composition in equation \eqref{comp muM smathcalA} is given on generators by
\begin{align*}
\gamma_i^M&\mapsto (i_*\gamma_i, [c_i^M])\\
\sigma^M&\mapsto (0,1).
\end{align*}
We can change the reference framing $\tau_M$ to a new framing $\tau_M'$ such that $\gamma_i^M$ changes to
\begin{equation*}
{\gamma_i^M}'=\gamma_i^M-c_i^M\sigma^M,
\end{equation*}
for all $i=1,\dotsc,2g$ and $\sigma^M$ stays the same. This change can be realized by a suitable self-diffeomorphism of $\partial\nu\Sigma_M$ according to the proof of Lemma \ref{act phi hom 1}. The composition in equation \eqref{comp muM smathcalA} now has the form  
\begin{align*}
{\gamma_i^M}'&\mapsto(i_*\gamma_i,0)\\
\sigma^M&\mapsto (0,1).
\end{align*}
\begin{lem} Suppose that $k_M>1$. There exists a trivialization $\tau_M$ of the normal bundle of $\Sigma_M$ in $M$, such that the composition 
\begin{equation*}
H_1(\partial M')\stackrel{\mu}{\rightarrow} H_1(M')\stackrel{s}{\rightarrow} H_1(M)\oplus \mathbb{Z}_{k_M}
\end{equation*}
is given by
\begin{align*}
\gamma_i^M&\mapsto(i_*\gamma_i,0),\quad i=1,\dotsc,2g \\
\sigma^M&\mapsto (0,1).
\end{align*}
A framing with this property is called {\em adapted}. If $k_M=1$ every framing is adapted.
\end{lem}
Every framing is adapted if $k_M=1$, since in this case $H_1(M')$ and $H_1(M)$ are isomorphic and $\sigma^M$ is null-homologous. We also have the following:
\begin{lem}\label{parallel curve adapted} Let $\tau_M$ be an adapted framing. Then a curve $\alpha$ on the surface $\Sigma_M$ is null-homologous in $M$ if and only if its parallel copy $\alpha^M$ on the push-off $\Sigma^M$ is null-homologous in $M'$.
\end{lem} 
We now derive the explicit expressions for the action of the map $\mu\colon \partial M'\rightarrow M'$ on cohomology and discuss rim tori in the manifold $M'$. These calculations will be useful later on.

\subsection{Calculation of the map $\mu^*\colon H^1(M')\rightarrow  H^1(\partial M')$.}

By Proposition \ref{calculation cohomology H^1M'}, the map $\rho^*\colon H^1(M)\rightarrow H^1(M')$ is an isomorphism. The framing $\tau_M$ defines an identification
\begin{equation*}
H^1(\partial M')\cong H^1(\Sigma)\oplus\mathbb{Z}PD(\Sigma^M),
\end{equation*}
where $\Sigma^M$ denotes the push-off of the surface $\Sigma_M$ and $PD(\Sigma^M)={\sigma^M}^*$. We want to derive a formula for the composition
\begin{equation*}
H^1(M)\cong H^1(M')\stackrel{\mu^*}{\longrightarrow} H^1(\partial M')\cong H^1(\Sigma)\oplus\mathbb{Z}PD(\Sigma^M).
\end{equation*}
Let $\alpha\in H^1(M)$. Then by the exactness of the Mayer-Vietoris sequence \eqref{eq Mayer Vietoris for adapted framings} 
\begin{align*}
\langle \mu^*\rho^*\alpha,\gamma_i^M\rangle &= \langle \alpha,\rho_*\mu_*\gamma_i^M\rangle \\
&=\langle \alpha, i_*\gamma_i\rangle \\
&=\langle i^*\alpha, \gamma_i\rangle,
\end{align*} 
and 
\begin{align*}
\langle \mu^*\rho^*\alpha,\sigma_i^M\rangle &= \langle \alpha,\rho_*\mu_*\sigma^M\rangle \\
&=0.
\end{align*} 
Hence we have:
\begin{lem}\label{Lemma muMrhoM H^1X}
The composition 
\begin{equation*}
\mu^*\circ\rho^*\colon H^1(M)\rightarrow  H^1(\Sigma)\oplus\mathbb{Z}PD(\Sigma^M)
\end{equation*}
is equal to $i^*\oplus 0$.
\end{lem}

\subsection{Calculation of the map $\mu^*\colon H^2(M')\rightarrow H^2(\partial M')$.}

Using the framing $\tau_M$ of the surface $\Sigma_M$ we can identify
\begin{equation*}
H^2(\partial M')\cong H_1(\Sigma\times S^1)\cong \mathbb{Z}\oplus H_1(\Sigma),
\end{equation*}
where the $\mathbb{Z}$ summand is spanned by $PD(\sigma^M)$. We can then consider the composition
\begin{equation*}
\left(H^2(M)/\mathbb{Z}\Sigma_M\right)\oplus \text{ker}\,i\cong H^2(M')\stackrel{\mu^*}{\rightarrow}H^2(\partial M') \cong \mathbb{Z}\oplus H_1(\Sigma).
\end{equation*}

\begin{prop}\label{adapted splitting for H2M prop} The composition 
\begin{equation}\label{mu splitting for H^2/Sigma}
\mu^*\colon \left(H^2(M)/\mathbb{Z}\Sigma_M\right)\oplus \text{ker}\,i \rightarrow \mathbb{Z}\oplus H_1(\Sigma)
\end{equation}
is given by 
\begin{equation*}
([V], \alpha)\mapsto (V\cdot\Sigma_M, \alpha).
\end{equation*}
\end{prop}
The map is well-defined in the first variable since $\Sigma_M^2=0$ and has image in $k_M\mathbb{Z}$, since $\Sigma_M$ is divisible by $k_M$. The map in the second variable is inclusion.
\begin{proof}
On the second summand, the map $\mu^*$ is the identity by Lemma \ref{eq poinc dual H2M'dM'} and the choice of splitting in Remark \ref{rem explicit splitting of H2M' as direct sum}. It remains to prove that
\begin{equation*}
\mu^*\rho^*[V]=(V\cdot\Sigma_M)PD(\sigma^M)
\end{equation*}
Exactness of the Mayer-Vietoris sequence for $M=M'\cup\nu \Sigma_M$ implies the equality $\mu^*\rho^*V=p^*i^*V$. Since
\begin{equation*}
\langle i^*V,\Sigma\rangle = \langle V,\Sigma_M\rangle=V\cdot\Sigma_M
\end{equation*}
the class $i^*V$ is equal to the class $(V\cdot\Sigma_M)1$, where $1$ denotes the generator of $H^2(\Sigma)$, Poincar\'e dual to a point. Since $p^*(1)$ is the Poincar\'e dual of a fibre in $\partial M'=\partial \nu\Sigma_M$, where $p$ denotes the projection $\partial \nu\Sigma_M\rightarrow \Sigma_M$, the claim follows.
\end{proof}

\begin{cor}\label{cor calc muM* P(M)_A_M} The composition
\begin{equation*}
\mu^*\colon P(M)_{A_M}\oplus \mathbb{Z}B_M\oplus \text{ker}\,i\longrightarrow \mathbb{Z}\oplus H_1(\Sigma)
\end{equation*}
is given by
\begin{equation*}
(c, xB_M, \alpha)\mapsto (xk_M, \alpha).
\end{equation*}
\end{cor}

\subsection{Calculation of $H_2(M')$.}

Using the homomorphism $\mu_M\colon H_2(\partial M')\rightarrow H_2(M')$ and the projection $p_M\colon \partial M'\rightarrow \Sigma$ we define a map $r_M$ given by 
\begin{equation*}
r_M=\mu_M\circ PD\circ p_M^*\colon H^1(\Sigma)\rightarrow H_2(M'). 
\end{equation*}  
The image of this homomorphism has the following interpretation: The map $PD\circ p_M^*$ determines an isomorphism of $H^1(\Sigma)$ onto $\text{ker}\,p_M$.  In our standard basis, this isomorphism is given by
\begin{equation}\label{rim torus for elements H^1}
\begin{split}
H^1(\Sigma)&\rightarrow \text{ker}\,p_M\\ 
\sum c_i\gamma_i^*&\mapsto \sum c_i\Gamma_i^M.
\end{split}
\end{equation}
\begin{lem} Every element in the image of $r_M$ can be represented by a smoothly embedded torus in the interior of $M'$.
\end{lem}
\begin{proof} Note that the classes $\Gamma_i^M\subset H_2(\partial M')$ are of the form $\chi_i^M\times \sigma^M$ where $\chi_i^M$ is a curve on $\Sigma_M$. Hence every element $T\in \text{ker}\,p_M$ is represented by a surface of the form $c^M\times \sigma^M$, where $c^M$ is a closed, oriented curve on $\Sigma_M$ with transverse self-intersections. A collar of $\partial M'=\partial\nu\Sigma_M$ in $M'$ is of the form $\Sigma_M\times S^1\times I$. We can eliminate the self-intersection points of the curve $c^M$ in $\Sigma_M\times I$, without changing the homology class. If we then take $c^M$ times $\sigma^M$, we see that $\mu_M(T)=c^M\times \sigma^M$ can be represented by a smoothly embedded torus in $M'$. 
\end{proof}

We make the following definition \cite{EP, FS3, IP}.
\begin{defn}\label{def rim tori 1} 
The map $r_M$ given by
\begin{equation}\label{defn eq rim tori map rM}
r_M=\mu_M\circ PD\circ p_M^*\colon H^1(\Sigma)\rightarrow H_2(M')
\end{equation} 
is called the {\em rim tori homomorphism} and the image of $r_M$, denoted by $R(M')$, the {\em group of rim tori} in $M'$. 
\end{defn}
Rim tori are already ``virtually'' in the manifold $M$ as embedded null-homologous tori. Some of them can become non-zero homology classes if the tubular neighbourhood $\nu\Sigma_M$ is deleted. The set of elements in $H^1(\Sigma)$ whose associated rim tori are null-homologous in $M'$ is given by the kernel of the rim tori map $r_M$, hence
\begin{equation}\label{rim tori related to ker rM}
R(M')\cong H^1(\Sigma)/\text{ker}\,r_M.
\end{equation}
We now derive a short exact sequence for the calculation of $H_2(M')$: Consider the following sequence coming from the long exact sequence for the pair $(M',\partial M')$:
\begin{equation}\label{long exact M', partial M'}
H_3(M',\partial M')\stackrel{\partial}{\rightarrow} H_2(\partial M')\stackrel{\mu}{\rightarrow} H_2(M')\rightarrow H_2(M',\partial M')\stackrel{\partial}{\rightarrow}H_1(\partial M').
\end{equation}
This induces a short exact sequence
\begin{equation}\label{eq short ex seq calc H2M' ker d}
0\longrightarrow H_2(\partial M')/\text{ker}\,\mu\stackrel{\mu}{\longrightarrow} H_2(M')\stackrel{}{\longrightarrow}\text{ker}\,\partial\longrightarrow 0.
\end{equation}

\begin{lem}\label{lem rim ker mu} The terms on the left and right side of the short exact sequence \eqref{eq short ex seq calc H2M' ker d} can be calculated as follows:
\begin{enumerate}
\item $\text{ker}\,\partial$ is isomorphic to 
\begin{equation*}
P(M)_{A_M}=H^2(M)/(\mathbb{Z}\Sigma_M\oplus\mathbb{Z}B_M).
\end{equation*}
\item The framing induces an isomorphism
\begin{equation*}
H_2(\partial M')/\text{ker}\,\mu\cong R(M')\oplus \mathbb{Z}PD(\Sigma^M).
\end{equation*}
In addition, the kernel of $r_M$ is equal to the image of $i_M^*\colon H^1(M)\rightarrow H^1(\Sigma)$. Hence the map $r_M$ induces an isomorphism $\text{coker}\,i_M^*\cong R(M')$.
\end{enumerate}
\end{lem}
\begin{proof}
Under Poincar\'e duality, the boundary homomorphism $\partial$ on the right hand side of the exact sequence \eqref{long exact M', partial M'} can be replaced by 
\begin{equation*}
\mu^*\colon H^2(M')\longrightarrow H^2(\partial M'),
\end{equation*}
as in Lemma \ref{eq poinc dual H2M'dM'}. Hence the kernel of $\partial$ can be replaced by the kernel of $\mu^*$ and claim (a) follows by Corollary \ref{cor calc muM* P(M)_A_M}.  

To prove part (b), consider the following part of the exact sequence \eqref{long exact M', partial M'}:
\begin{equation}
H_3(M',\partial M')\stackrel{\partial}{\longrightarrow} H_2(\partial M')\stackrel{\mu}{\longrightarrow} H_2(M').
\end{equation}
Under Poincar\'e duality and the isomorphism $\rho^*\colon H^1(M)\rightarrow H^1(M')$ this sequence becomes
\begin{equation}
H^1(M)\stackrel{\mu^*\circ\rho^*}{\longrightarrow} H^1(\partial M')\stackrel{\mu\circ PD}{\longrightarrow} H_2(M').
\end{equation}
By Lemma \ref{Lemma muMrhoM H^1X} this sequence can be written as
\begin{equation}\label{seq calc H2dM'/ker muM}
H^1(M)\stackrel{i^*\oplus 0}{\longrightarrow} H^1(\Sigma)\oplus\mathbb{Z}PD(\Sigma^M)\stackrel{r_M+\mu}{\longrightarrow} H_2(M'),
\end{equation}
where $r_M$ denotes the homomorphism above and $\mu$ is the restriction to $\mathbb{Z}PD(\Sigma^M)$. Exactness of sequence \eqref{seq calc H2dM'/ker muM} implies that the kernel of the homomorphism $r_M$ is equal to the image of $i^*$.
\end{proof}
Together with equation \eqref{rim tori related to ker rM} we get:
\begin{prop} The short exact sequence \eqref{eq short ex seq calc H2M' ker d} for the calculation of $H_2(M')$ can be written as
\begin{equation*}
0\longrightarrow R(M')\oplus \mathbb{Z}PD(\Sigma^M)\longrightarrow H_2(M')\longrightarrow P(M)_{A_M}\longrightarrow 0.
\end{equation*}
\end{prop}

\section{Calculation of $H_1(X)$ and $H^1(X)$}\label{section calculation H_1X, H^1X}

Let $X=M \#_{\Sigma_M=\Sigma_N} N$ denote the generalized fibre sum of two closed 4-manifolds $M$ and $N$ along embedded surfaces $\Sigma_M$ and $\Sigma_N$ of genus $g$ and self-intersection zero. We assume as in Section \ref{section hom cohom of M'} that $\Sigma_M$ and $\Sigma_N$ represent non-torsion classes of maximal divisibility $k_M$ and $k_N$ and choose adapted framings for both surfaces as defined in Section \ref{subsect adapted framings}.

Consider the homomorphisms
\begin{align*}
i_M\oplus i_N\colon H_1(\Sigma;\mathbb{Z})&\longrightarrow H_1(M;\mathbb{Z})\oplus H_1(N;\mathbb{Z})\\
\lambda &\mapsto (i_M(\lambda),i_N(\lambda)),
\end{align*}
and
\begin{align*}
i_M^*+i_N^*\colon H^1(M;\mathbb{Z})\oplus H^1(N;\mathbb{Z})&\longrightarrow H^1(\Sigma;\mathbb{Z})\\
(\alpha,\beta)&\mapsto i_M^*\alpha+i_N^*\beta.
\end{align*}
The kernels of $i_M\oplus i_N$ and $i_M^*+i_N^*$ are free abelian groups, but the cokernels may have torsion. Both homomorphisms can also be considered for homology and cohomology with $\mathbb{R}$-coefficients. 

\begin{defn}\label{defn dimension d of ker iM*+iN*} Let $d$ denote the integer $d=\text{dim ker}\,(i_M\oplus i_N)$ for the linear map
\begin{equation*}
i_M\oplus i_N\colon H_1(\Sigma;\mathbb{R})\longrightarrow H_1(M;\mathbb{R})\oplus H_1(N;\mathbb{R})
\end{equation*}
of $\mathbb{R}$-vector spaces.
\end{defn}

\begin{lem}\label{d Ker Coker} Consider the homomorphisms $i_M\oplus i_N$ and $i_M^*+i_N^*$ for homology and cohomology with $\mathbb{R}$-coefficients. Then 
\begin{align*}
\text{dim ker}\,(i_M^*+i_N^*)&= b_1(M)+b_1(N)-2g +d = \text{dim coker}\, (i_M\oplus i_N)\\
\text{dim coker}\,(i_M^*+i_N^*)&= d = \text{dim ker}\,(i_M\oplus i_N),
\end{align*}
where $g$ denotes the genus of the surface $\Sigma$. 
\end{lem}
\begin{proof} 
By linear algebra, $i_M^*+i_N^*$ is the dual homomorphism to $i_M\oplus i_N$ under the identification of cohomology with the dual vector space of homology with $\mathbb{R}$-coefficients. Moreover,
\begin{align*}
\text{dim coker}\,(i_M\oplus i_N)&=b_1(M)+b_1(N)-\text{dim im}\,(i_M\oplus i_N)\\
&= b_1(M)+b_1(N)-(2g-\text{dim ker}\,(i_M\oplus i_N))\\
&= b_1(M)+b_1(N)-2g+d.
\end{align*}
This implies
\begin{align*}
\text{dim ker}\,(i_M^*+i_N^*)&= \text{dim coker}\,(i_M\oplus i_N)=b_1(M)+b_1(N)-2g +d\\
\text{dim coker}\,(i_M^*+i_N^*)&= \text{dim ker}\,(i_M\oplus i_N)=d.
\end{align*}
\end{proof}
In the following calculations we will often use the Mayer-Vietoris sequence associated to the decomposition $X=M'\cup N'$, given by
\begin{equation*}
\ldots\rightarrow H_k(\partial M')\stackrel{\psi_k}{\rightarrow} H_k(M')\oplus H_k(N')\rightarrow H_k(X)\rightarrow H_{k-1}(\partial M')\rightarrow\ldots 
\end{equation*}
with homomorphisms
\begin{align*}
\psi_k\colon H_k(\partial M')\rightarrow H_k(M')\oplus H_k(N'),&\quad \alpha\mapsto (\mu_M\alpha,\mu_N\phi_*\alpha) \\
H_k(M')\oplus H_k(N')\rightarrow H_k(X),&\quad (x,y)\mapsto \eta_M x - \eta_N y.
\end{align*} 
We also use the Mayer-Vietoris sequence for cohomology groups.

\subsection{Calculation of $H^1(X)$}
We begin with the calculation of the first cohomology of $X$. Consider the following part of the Mayer-Vietoris sequence in cohomology:
\begin{equation*}
0\rightarrow H^1(X)\stackrel{\eta_M^*\ominus\eta_N^*}{\longrightarrow}H^1(M')\oplus H^1(N')\stackrel{\psi_1^*}{\longrightarrow}H^1(\partial M').
\end{equation*}
Since $\eta_M^*-\eta_N^*$ is injective, $H^1(X)$ is isomorphic to the kernel of $\psi_1^*=\mu_M^*+\phi^*\mu_N^*$. Composing with isomorphisms, the map $\psi_1^*$ can be replaced by the map $\mu_M^*\rho_M^*+\phi^*\mu_N^*\rho_N^*$. Since the equality $\phi^*{\gamma_i^N}^*={\gamma_i^M}^*$ holds for all indices $i$, it follows with Lemma \ref{Lemma muMrhoM H^1X} that this composition can be replaced by the map
\begin{equation}\label{iM*+iN* replaced}
(i_M^*+i_N^*)\oplus 0\colon H^1(M)\oplus H^1(N)\rightarrow H^1(\Sigma)\oplus\mathbb{Z}PD(\Sigma^M).   
\end{equation}
This implies:
\begin{thm}\label{prop first cohom X} The first cohomology $H^1(X;\mathbb{Z})$ is isomorphic to the kernel of
\begin{equation*}
i_M^*+i_N^*\colon H^1(M;\mathbb{Z})\oplus H^1(N;\mathbb{Z})\rightarrow H^1(\Sigma;\mathbb{Z}).
\end{equation*}
\end{thm}
As a corollary we can calculate the Betti numbers of $X$. 
\begin{cor}\label{betti X} The Betti numbers of a generalized fibre sum $X=M\#_{\Sigma_M=\Sigma_N}N$ along surfaces $\Sigma_M$ and $\Sigma_N$ of genus $g$ and self-intersection zero are given by
\begin{align*}
b_0(X)&=b_4(X)=1\\
b_1(X)&=b_3(X)=b_1(M)+b_1(N)-2g+d\\
b_2(X)&=b_2(M)+b_2(N)-2+2d\\
b_2^+(X)&=b_2^+(M)+b_2^+(N)-1+d\\
b_2^-(X)&=b_2^-(M)+b_2^-(N)-1+d,
\end{align*}
where $d$ is the integer from Definition \ref{defn dimension d of ker iM*+iN*}. 
\end{cor}
\begin{proof} The formula for $b_1(X)$ follows from Theorem \ref{prop first cohom X} and Lemma \ref{d Ker Coker}. To derive the formula for $b_2(X)$, we use the formula for the Euler characteristic of a space decomposed into two parts $A,B$:
\begin{equation*}
e(A\cup B)=e(A)+e(B)-e(A\cap B),
\end{equation*}
For $M=M'\cup\nu\Sigma_M$, with $M'\cap \nu\Sigma_M\cong\Sigma\times S^1$, we get
\begin{align*}
e(M)&=e(M')+e(\nu\Sigma_M)-e(\Sigma\times S^1)\\
&=e(M')+2-2g,
\end{align*}
since $\nu\Sigma_M$ is homotopy equivalent to $\Sigma_M$ and $\Sigma\times S^1$ is a 3-manifold, hence has zero Euler characteristic. This implies 
\begin{equation*}
e(M')=e(M)+2g-2,\quad\text{and similarly}\quad e(N')=e(N)+2g-2.
\end{equation*}
For $X=M'\cup N'$, with $M'\cap N'\cong\Sigma\times S^1$, we then get
\begin{align*}
e(X)&=e(M')+e(N')\\
&=e(M)+e(N)+4g-4.
\end{align*}
Substituting the formula for $b_1(X)=b_3(X)$ above, this implies
\begin{align*}
b_2(X)&=-2+2(b_1(M)+b_1(N)-2g+d)+2-2b_1(M)+b_2(M)\\
&\quad +2-2b_1(N)+b_2(N)+4g-4\\
&= b_2(M)+b_2(N)-2+2d.
\end{align*}
It remains to prove the formula for $b_2^{\pm}(X)$. By Novikov additivity for the signature \cite[Remark 9.1.7]{GS},
\begin{equation*}
\sigma(X)=\sigma(M)+\sigma(N),
\end{equation*}
we get by adding $b_2(X)$ on both sides,
\begin{equation*}
2b_2^+(X)=2b_2^+(M)+2b_2^+(N)-2+2d,
\end{equation*}
hence $b_2^+(X)=b_2^+(M)+b_2^+(N)-1+d$. This also implies the formula for $b_2^-(X)$. 
\end{proof}
A direct computation of $b_2(X)$ as the rank of $H_2(X)$ will be given in Section \ref{subsec Calculation of H2(X)}.

\subsection{Calculation of $H_1(X)$}

In this subsection we prove a formula for the first integral homology of $X$. We make the following definition:
\begin{defn}
Let $n_{MN}$ denote the greatest common divisor of $k_M$ and $k_N$.
\end{defn}
If $n_{MN}$ is not equal to $1$, the formula for $H_1(X)$ involves an additional torsion term. Let $r$ denote the homomorphism defined by
\begin{align*}
r\colon H_1(\Sigma;\mathbb{Z})&\longrightarrow \mathbb{Z}_{n_{MN}},\\
\lambda&\mapsto \langle C,\lambda\rangle \mod n_{MN}.
\end{align*} 
We then have the following theorem:

\begin{thm}\label{H 1 for X} Consider the homomorphism
\begin{align*}
H_1(\Sigma;\mathbb{Z})&\stackrel{i_M\oplus i_N\oplus r}{\longrightarrow} H_1(M;\mathbb{Z})\oplus H_1(N;\mathbb{Z})\oplus \mathbb{Z}_{n_{MN}},\\
\lambda &\mapsto  (i_M \lambda,i_N \lambda, r(\lambda)).
\end{align*}
Then $H_1(X;\mathbb{Z})$ is isomorphic to the cokernel of $i_M\oplus i_N\oplus r$.
\end{thm}
\begin{proof} Since $\partial M'$ is connected, the Mayer-Vietoris sequence for $X$ shows that
\begin{equation*}
H_1(X)\cong \text{coker}(\psi_1\colon H_1(\partial M')\rightarrow H_1(M')\oplus H_1(N')).
\end{equation*}
The homomorphism $\psi_1$ is given on the standard basis by
\begin{align*}
\gamma_i^M&\mapsto (\mu_M\gamma_i^M,\mu_N\gamma_i^N+a_i\mu_N\sigma^N)\\
\sigma^M&\mapsto (\mu_M\sigma^M,-\mu_N\sigma^N).
\end{align*}
We want to replace $H_1(M')$ by $H_1(M)\oplus \mathbb{Z}_{k_M}$ and $H_1(N')$ by $H_1(N)\oplus \mathbb{Z}_{k_N}$, as in Proposition \ref{H_1 complement}. We choose isomorphisms $s$ as in Section \ref{subsect adapted framings}. Since we are working with adapted framings, the composition
\begin{equation}
H_1(\partial M')\stackrel{\mu_M}{\rightarrow} H_1(M')\stackrel{s}{\rightarrow} H_1(M)\oplus \mathbb{Z}_{k_M}
\end{equation}
is given on generators by
\begin{align*}
\gamma_i^M&\mapsto (i_M\gamma_i, 0)\\
\sigma^M&\mapsto (0,1),
\end{align*}
as before. An analogous map exists for $N$. If we add these maps together, the homomorphism $\psi_1$ can be replaced by
\begin{align*}
H_1(\partial M')&\rightarrow H_1(M)\oplus \mathbb{Z}_{k_M}\oplus H_1(N)\oplus \mathbb{Z}_{k_N},\\
\gamma_i^M&\mapsto (i_M\gamma_i,0,i_N\gamma_i,a_i)\\
\sigma^M&\mapsto (0,1,0,-1).
\end{align*}
Using the isomorphism $H_1(\Sigma\times S^1)\cong H_1(\Sigma)\oplus \mathbb{Z}\rightarrow H_1(\partial M')$ given by the framing $\tau_M$, we get the map
\begin{equation}\label{psi 1}
\begin{split}
H_1(\Sigma)\oplus\mathbb{Z}&\rightarrow H_1(M)\oplus \mathbb{Z}_{k_M}\oplus H_1(N)\oplus \mathbb{Z}_{k_N},\\
(\lambda,\alpha)&\mapsto (i_M\lambda,\text{$\alpha$ mod $k_M$},i_N\lambda,\text{$\langle C,\lambda\rangle -\alpha$ mod $k_N$}).
\end{split}
\end{equation}
To finish the proof, we have to show that this map has the same cokernel as the map
\begin{align*}
i_M\oplus i_N\oplus r\colon H_1(\Sigma)&\rightarrow H_1(M)\oplus H_1(N)\oplus \mathbb{Z}_{n_{MN}},\\
\lambda&\mapsto (i_M\lambda,i_N\lambda,\text{$\langle C,\lambda\rangle$ mod $n_{MN}$}).
\end{align*}
This follows from Lemma \ref{lemma for comp H1} below.
\end{proof}

In the proof we used a small algebraic lemma which can be formulated as follows: Let $H$ and $G$ be abelian groups and $f\colon H\rightarrow G$ and $h\colon H\rightarrow \mathbb{Z}$ homomorphisms. Let $k_M,k_N$ be positive integers with greatest common divisor $n_{MN}$. Consider the (well-defined) map 
\begin{align*}
p\colon \mathbb{Z}_{k_M}\oplus \mathbb{Z}_{k_N}&\rightarrow \mathbb{Z}_{n_{MN}},\\
([x],[y])&\mapsto [x+y].
\end{align*}  
\begin{lem}\label{lemma for comp H1} The homomorphisms
\begin{align*}
\psi\colon H\oplus \mathbb{Z}&\rightarrow G\oplus \mathbb{Z}_{k_M}\oplus \mathbb{Z}_{k_N}\\
(x,a)&\mapsto (f(x), \text{$a$ mod $k_M$}, \text{$h(x)-a$ mod $k_N$}),
\end{align*}
and 
\begin{align*}
\psi'\colon H &\rightarrow G\oplus \mathbb{Z}_{n_{MN}}\\
x&\mapsto (f(x), \text{$h(x)$ mod $n_{MN}$})
\end{align*}
have isomorphic cokernels. The isomorphism is induced by $\text{Id}_G\oplus p$.
\end{lem} 
\begin{proof} The map $\text{Id}_G\oplus p$ is a surjection, hence it induces a surjection 
\begin{equation*}
P\colon G\oplus \mathbb{Z}_{k_M}\oplus \mathbb{Z}_{k_N}\rightarrow \text{coker}\,\psi'.
\end{equation*}
We compute the kernel of $P$ and show that it is equal to the image of $\psi$. This will prove the lemma. Suppose an element is in the image of $\psi$. Then it is of the form $(f(x), \text{$a$ mod $k_M$}, \text{$h(x)-a$ mod $k_N$})$. The image under $\text{Id}_G\oplus p$ of this element is equal to $(f(x),\text{$h(x)$ mod $n_{MN}$})$, hence in the image of $\psi'$. Conversely, let $(g, \text{$u$ mod $k_M$}, \text{$v$ mod $k_N$})$ be an element in the kernel of $P$. The element maps under $\text{Id}_G\oplus p$ to $(g, \text{$u+v$ mod $n_{MN}$})$, hence there exists an element $x\in H$ such that $g=f(x)$ and $u+v\equiv h(x)$ mod $n_{MN}$. We can choose integers $c,d,e$ such that the following equations hold:
\begin{equation*}
u+v-h(x)=cn_{MN}=dk_M+ek_N.
\end{equation*}
Define an integer $a=u-dk_M$. Then:
\begin{align*}
u&\equiv \text{$a$ mod $k_M$}\\
v&\equiv \text{$h(x)-a+ek_N=h(x)-a$ mod $k_N$}.
\end{align*}
Hence $(g,\text{$u$ mod $k_M$}, \text{$v$ mod $k_N$})=\psi(x,a)$ and the element is in the image of $\psi$. 
\end{proof} 
\begin{rem} Note that Theorem \ref{H 1 for X} only holds if the cohomology class $C$ is calculated in adapted framings. For a given gluing diffeomorphism, the cohomology class $C$ depends on the choice of adapted framings. However, the result $H_1(X;\mathbb{Z})$ does not depend on the choice of framings, only on the gluing diffeomorphism itself. This follows because a change of adapted framings changes the coefficients $a_i=\langle C,\gamma_i\rangle$ to
\begin{equation*}
a_i'=\langle C',\gamma_i\rangle=a_i-\beta_i^Mk_M-\beta_i^Nk_N,
\end{equation*}
where $\beta_i^M,\beta_i^N$ are certain integers. This implies that the map $r$ defined above stays the same.
\end{rem}
An immediate corollary of Theorem \ref{H 1 for X} is the following.
\begin{cor}\label{cor H_1 n_MN=1} If the divisibilities $k_M$ and $k_N$ are coprime, then $H_1(X;\mathbb{Z})$ is isomorphic to the cokernel of $i_M\oplus i_N\colon H_1(\Sigma;\mathbb{Z}) \longrightarrow  H_1(M;\mathbb{Z})\oplus H_1(N;\mathbb{Z})$.
\end{cor}

\section{Calculation of $H^2(X)$}\label{sect second cohom Gompf sum}

The computation of $H^2(X)$ is based on the following lemma: 
\begin{lem}\label{lem short exact seq for cal H^2X} The following part of the Mayer-Vietoris sequence 
\begin{equation*}
H^1(M')\oplus H^1(N')\stackrel{\psi_1^*}{\rightarrow}H^1(\partial M')\rightarrow H^2(X)\stackrel{\eta_M^*\ominus \eta_N^*}{\longrightarrow}H^2(M')\oplus H^2(N')\stackrel{\psi_2^*}{\rightarrow}H^2(\partial M').
\end{equation*}
induces a short exact sequence
\begin{equation}\label{ses for H^2}
0\longrightarrow \text{coker}\,\psi_1^*\longrightarrow H^2(X)\longrightarrow \text{ker}\,\psi_2^*\longrightarrow 0.
\end{equation}
\end{lem}
By equation \eqref{iM*+iN* replaced} there exists an isomorphism 
\begin{equation}\label{coker psi1 term}
\text{coker}\,\psi_1^*\cong\text{coker}\,(i_M^*+i_N^*) \oplus \mathbb{Z}PD(\Sigma^M).
\end{equation}
We calculate $\text{coker}\,(i_M^*+i_N^*)$ in subsection \ref{subsection rim tori in X} and $\text{ker}\,\psi_2^*$ in subsection \ref{subsection vanishing classes}.

\subsection{Rim tori}\label{subsection rim tori in X}

We can map every rim torus in $M'$ under the inclusion $\eta_M\colon M'\rightarrow X$ to a homology class in $X$.
\begin{defn}\label{def rim tori 2} We call $\eta_M\circ r_M(\alpha)$ the rim torus in $X$ associated to the element $\alpha \in H^1(\Sigma)$ via $M'$. The group $R(X)$ of rim tori in $X$ is defined as the image of the homomorphism
\begin{equation*}
r_X=\eta_M\circ r_M\colon H^1(\Sigma)\rightarrow H_2(X).
\end{equation*} 
Similarly, we can map every rim torus in $N'$ under the inclusion $\eta_N\colon N'\rightarrow X$ to a homology class in $X$. This defines a homomorphism 
\begin{equation*}
r_X'=\eta_N\circ r_N\colon H^1(\Sigma)\rightarrow H_2(X).
\end{equation*} 
\end{defn}
The rim tori in $X$ coming via $M'$ and $N'$ are related in the following way:
\begin{lem}\label{rim from N is minus M} Let $\alpha$ be a class in $H^1(\Sigma)$. Then $r_X(\alpha)=-r_X'(\alpha)$. Hence for the same element $\alpha\in H^1(\Sigma)$ the rim torus in $X$ coming via $N'$ is minus the rim torus coming via $M'$.
\end{lem}
\begin{proof}
The action of the gluing diffeomorphism $\phi$ on second homology is given by $\phi_*\Gamma_i^M=-\Gamma_i^N$. Let $\alpha\in H^1(\Sigma)$ be a fixed class,
\begin{equation*}
\alpha=\sum_{i=1}^{2g}c_i\gamma_i^*.
\end{equation*}
The rim tori in $M'$ and $N'$ associated to $\alpha$ are given by
\begin{equation*}
a_M=\sum_{i=1}^{2g}c_i\mu_M\Gamma_i^M,\quad a_N=\sum_{i=1}^{2g}c_i\mu_N\Gamma_i^N=-\sum_{i=1}^{2g}c_i\mu_N\phi_*\Gamma_i^M.
\end{equation*}
In $X$ we get
\begin{align*}
\eta_Ma_M+\eta_Na_N &=\sum_{i=1}^{2g}c_i(\eta_M\mu_M-\eta_N\mu_N\phi_*)\Gamma_i^M\\
&=0,
\end{align*}
by the Mayer-Vietoris sequence for $X$. This proves the claim.
\end{proof}

\begin{defn}\label{defn rim tori C} Let $R_C$ denote the rim torus in $X$ determined by the class 
\begin{equation*}
-\sum_{i=1}^{2g}a_i\Gamma_i^M\in H_2(\partial M')
\end{equation*}
under the inclusion of $\partial M'$ in $X$ as in Definition \ref{def rim tori 2}. Here $a_i$ are the integers from Definition \ref{defn coeff ai}. This class is equal to the image of the class $\sum_{i=1}^{2g}a_i\Gamma_i^N\in H_2(\partial N')$ under the inclusion of $\partial N'$ in $X$.
\end{defn}
Recall that $\Sigma_X$ is the class in $X$ which is the image of the push-off $\Sigma^M$ under the inclusion $M'\rightarrow X$. Similarly, $\Sigma_X'$ is the image of the push-off $\Sigma^N$ under the inclusion $N'\rightarrow X$. 
\begin{lem}\label{lem difference sigmaXX' RC} The classes $\Sigma_X'$ and $\Sigma_X$ in $X$ differ by  
\begin{equation*}
\Sigma_X'-\Sigma_X=R_C.
\end{equation*}
\end{lem}
\begin{proof}
This follows, since by Lemma \ref{computation action phi H_2}
\begin{equation*}
\phi_*\Sigma^M=-\left(\sum_{i=1}^{2g}a_i\Gamma_i^N\right)+\Sigma^N.
\end{equation*}
By the Mayer-Vietoris sequence for $X$ we get
\begin{equation*}
\Sigma_X=\eta_M\mu_M\Sigma^M=\eta_N\mu_N\phi_*\Sigma^M=-R_C+\Sigma_X'.
\end{equation*}
\end{proof}
The difference is due to the fact that the diffeomorphism $\phi$ does not necessarily match the classes $\Sigma^M$ and $\Sigma^N$. We now prove the main theorem in this subsection. 

\begin{thm}\label{isom coker R rim} Let $i_M^*, i_N^*$ denote the homomorphisms
\begin{equation*}
i_M^*\colon H^1(M;\mathbb{Z})\rightarrow H^1(\Sigma;\mathbb{Z}),\quad\text{and}\quad i_N^*\colon H^1(N;\mathbb{Z})\rightarrow H^1(\Sigma;\mathbb{Z}).
\end{equation*}
Then the kernel of the rim tori map $r_X$ is equal to the image of $i_M^*+i_N^*$. Hence the map $r_X$ induces an isomorphism
\begin{equation*}
\text{coker}\,(i_M^*+i_N^*)\stackrel{\cong}{\longrightarrow} R(X).
\end{equation*}
\end{thm}
\begin{lem}\label{lem ker etaM is im rMiN*} The kernel of the map $\eta_M$ is equal to the image of $r_M\circ i_N^*$.
\end{lem}
\begin{proof} Consider the following diagram:
\begin{equation*}
\begin{CD}
H_3(X,M') @>\partial >> H_2(M') @>\eta_M >> H_2(X)\\
@A\cong AA @A\mu_M\phi_*^{-1} AA @A\eta_N AA\\
H_3(N',\partial N') @>\partial >> H_2(\partial N') @>>> H_2(N')
\end{CD}
\end{equation*}
The horizontal parts come from the long exact sequences of pairs, the vertical parts come from inclusion. The isomorphism on the left is by excision. Hence the kernel of $\eta_M$ is given by the image of
\begin{equation*}
\mu_M\circ\phi_*^{-1}\circ\partial\colon H_3(N',\partial N')\rightarrow H_2(M').
\end{equation*}
Given the definition of the rim tori map $r_M$ in equation \eqref{defn eq rim tori map rM}, we have to show that this is equal to the image of 
\begin{equation*}
\mu_M\circ PD\circ p_M^*\circ i_N^*\colon H^1(N)\rightarrow H_2(M').
\end{equation*}
This follows in three steps: First, by Lemma \ref{eq poinc dual H2M'dM'} and Proposition \ref{calculation cohomology H^1M'}, the image of $\partial$ is equal to the image of $PD\circ \mu_N^*\circ \rho_N^*$. By the Mayer-Vietoris sequence for $N$
\begin{equation*}
H^1(N)\stackrel{\rho_N^*\ominus i_N^*}{\longrightarrow}H^1(N')\oplus H^1(\Sigma)\stackrel{\mu_N^*+p_N^*}{\longrightarrow}H^1(\partial N')
\end{equation*}
we have $\mu_N^*\circ \rho_N^*=p_N^*\circ i_N^*$. Finally, we use the identity
\begin{equation*}
\phi_*^{-1}\circ PD\circ p_N^*=-PD\circ p_M^*,
\end{equation*}
which is equivalent to the identity $\phi_*\Gamma_i^M=-\Gamma_i^N$, for all $i=1,\dotsc,2g$, from Lemma \ref{computation action phi H_2}.     
\end{proof}
We can now prove Theorem \ref{isom coker R rim}.
\begin{proof}
Suppose that $\alpha\in H^1(\Sigma)$ is in the kernel of $r_X=\eta_M\circ r_M$. This happens if and only if $r_M(\alpha)$ is in the kernel of $\eta_M$. By Lemma \ref{lem ker etaM is im rMiN*} this is equivalent to the existence of a class $\beta_N\in H^1(N)$ with
\begin{equation*}
r_M(\alpha)=(r_M\circ i_N^*)(\beta_N).
\end{equation*}
Since the kernel of $r_M$ is equal to the image of $i_M^*$ by Lemma \ref{lem rim ker mu}, this is equivalent to the existence of a class $\beta_M\in H^1(M)$ with 
\begin{equation*}
\alpha=i_M^*\beta_M+i_N^*\beta_N.
\end{equation*}
This shows that the kernel of $r_X$ is equal to the image of $i_M^*+i_N^*$ and proves the claim.
\end{proof}

\begin{cor} The rank of the abelian subgroup $R(X)$ of rim tori in $X$ is equal to the integer $d$ from Definition \ref{defn dimension d of ker iM*+iN*}.
\end{cor} 
Note that the rim tori group can contain torsion elements.

\subsection{Vanishing classes}\label{subsection vanishing classes} For the calculation of $H^2(X)$ it remains to calculate the kernel of
\begin{equation*}
\psi_2^*\colon H^2(M')\oplus H^2(N')\rightarrow H^2(\partial M'),
\end{equation*}
where $\psi_2^*=\mu_M^*+\phi^*\mu_N^*$, as in equation \eqref{ses for H^2}. We will first replace this map by an equivalent map. By Corollary \ref{cor calc muM* P(M)_A_M} and Lemma \ref{computation action phi H^2} we can replace the map
\begin{equation*}
\phi^*\colon H^2(\partial N')\rightarrow H^2(\partial M')
\end{equation*}
by the map
\begin{align*}
\mathbb{Z}\oplus H_1(\Sigma)&\longrightarrow \mathbb{Z}\oplus H_1(\Sigma)\\
(x,y)&\mapsto (x-\langle C,y\rangle, -y).   
\end{align*}
We then get:
\begin{lem} The map $\psi_2^*=\mu_M^*+\phi^*\mu_N^*$ can be replaced by the homomorphism
\begin{equation*}
P(M)_{A_M}\oplus P(N)_{A_N}\oplus \mathbb{Z}B_M\oplus\mathbb{Z}B_N \oplus\text{ker}\,i_M\oplus\text{ker}\,i_N\longrightarrow  \mathbb{Z}\oplus H_1(\Sigma)
\end{equation*}
given by
\begin{equation*}
(c_M, c_N, x_M, x_N, \alpha_M, \alpha_N) \mapsto  (x_Mk_M+x_Nk_N-\langle C,\alpha_N \rangle, \alpha_M-\alpha_N).
\end{equation*}
\end{lem}

\begin{defn}\label{defn general vanishing classes} We consider the map
\begin{align*}
f\colon \mathbb{Z}B_M\oplus\mathbb{Z}B_N\oplus \text{ker}\,(i_M\oplus i_N)&\longrightarrow\mathbb{Z}\\
(x_MB_M,x_NB_N,\alpha)&\mapsto x_Mk_M+x_Nk_N-\langle C,\alpha\rangle.
\end{align*}
Let $S(X)$ denote the kernel of the map $f$. We call $S(X)$ the group of {\em vanishing classes} of $X$, as in \cite{FSfam}. It is a free abelian group of rank $d+1$ by Lemma \ref{d Ker Coker}. 
\end{defn}
These classes have the following interpretation:
\begin{lem}\label{lem interpretation S(X)} The elements $(x_MB_M,x_NB_N,\alpha)$ in $S(X)$ are precisely those elements in $\mathbb{Z}B_M\oplus\mathbb{Z}B_N\oplus \text{ker}\,(i_M\oplus i_N)$ such that $\alpha^M+x_Mk_M\sigma^M$ bounds in $M'$, $\alpha^N+x_Nk_N\sigma^N$ bounds in $N'$, and both elements get identified under the gluing diffeomorphism $\phi$. 
\end{lem}
\begin{proof} Let $\alpha\in\text{ker}\,(i_M\oplus i_N)$ and $x_Mk_M+x_Nk_N-\langle C,\alpha\rangle=0$. Since we are using adapted framings the parallel curves $\alpha^M$ and $\alpha^N$ are null-homologous in $M'$ and $N'$. Then $\alpha^M+x_Mk_M\sigma^M$ bounds in $M'$ and $\alpha^N+x_Nk_N\sigma^N$ bounds in $N'$. Under the gluing diffeomorphism $\phi$ the first curve maps to
\begin{equation*}
\alpha^N+\langle C,\alpha\rangle \sigma^N-x_Mk_M\sigma^N.
\end{equation*}
This is equal to the second curve by our assumption. This proves one direction of the claim. The converse follows similarly. 
\end{proof}
We can now prove:
\begin{thm}\label{ker psi2 term} The kernel of the homomorphism 
\begin{equation*}
\psi_2^*\colon H^2(M')\oplus H^2(N')\rightarrow H^2(\partial M')
\end{equation*}
is isomorphic to $S(X)\oplus P(M)_{A_M}\oplus P(N)_{A_N}$.
\end{thm}
\begin{proof}
Elements in the kernel must satisfy $\alpha_M=\alpha_N$. In particular, both elements are in $\text{ker}\,i_M\cap \text{ker}\,i_N=\text{ker}\,(i_M\oplus i_N)$. Hence the kernel of the replaced $\psi_2^*$ is given by
\begin{equation*}
S(X)\oplus P(M)_{A_M}\oplus P(N)_{A_N}.
\end{equation*}
\end{proof}

\subsection{Calculation of $H^2(X)$}\label{subsec Calculation of H2(X)}
Using Theorem \ref{ker psi2 term} and Theorem \ref{isom coker R rim} we can now replace the terms in Lemma \ref{lem short exact seq for cal H^2X} to get the following theorem:
\begin{thm}\label{computation thm H2} There exists a short exact sequence
\begin{equation*}
0\rightarrow R(X)\oplus \mathbb{Z}\Sigma_X\rightarrow H^2(X;\mathbb{Z})\rightarrow S(X)\oplus P(M)_{A_M}\oplus P(N)_{A_N}\rightarrow 0.
\end{equation*}
\end{thm}
Since we already calculated the rank of each group occuring in this short exact sequence, we can calculate the second Betti number of $X$: 
\begin{align*}
b_2(X)&=d+1+(d+1)+(b_2(M)-2)+(b_2(N)-2)\\
&= b_2(M)+b_2(N)-2+2d.
\end{align*}
This is the same number as in Corollary \ref{betti X}. 

In the following sections we will restrict to the case that $\Sigma_M, \Sigma_N$ represent indivisible classes, hence $k_M=k_N=1$, and the cohomologies of $M$, $N$ and $X$ are torsion-free, which is equivalent to $H^2$ or $H_1$ being torsion free. For the manifold $X$ this can be checked using the formula for $H_1(X)$ in Theorem \ref{H 1 for X} or Corollary \ref{cor H_1 n_MN=1}. Then the exact sequence in Theorem \ref{computation thm H2} reduces to
\begin{equation*}
0\rightarrow R(X)\oplus \mathbb{Z}\Sigma_X\rightarrow H^2(X)\rightarrow S(X)\oplus P(M)\oplus P(N)\rightarrow 0.
\end{equation*}
Since all groups are free abelian, the sequence splits and we get an isomorphism
\begin{equation*}
H^2(X)\cong P(M)\oplus P(N)\oplus S(X)\oplus R(X)\oplus \mathbb{Z}\Sigma_X.
\end{equation*}

\section{The intersection form of $X$}\label{section calc intersection form of X}

From now on until the end of this article we will assume that $\Sigma_M$ and $\Sigma_N$ represent indivisible classes and the cohomologies of $M$, $N$ and $X$ are torsion free.
The group of vanishing classes $S(X)$ contains the element
\begin{equation*}
B_X=B_M-B_N.
\end{equation*}
We want to prove that we can choose $d$ elements $S_1,\dotsc,S_d$ in $S(X)$, forming a basis for a subgroup $S'(X)$ such that $S(X)=\mathbb{Z}B_X\oplus S'(X)$, and a basis $R_1,\dotsc,R_d$ for the group of rim tori $R(X)$ such that the following holds:

\begin{thm}\label{formula H^2 no torsion} Let $X=M\#_{\Sigma_M=\Sigma_N}N$ be a generalized fibre sum of closed oriented 4-manifolds $M$ and $N$ along embedded surfaces $\Sigma_M,\Sigma_N$ of genus $g$ and self-intersection zero which represent indivisible homology classes. Suppose that the cohomology of $M$, $N$ and $X$ is torsion free. Then there exists a splitting
\begin{equation*}
H^2(X;\mathbb{Z})=P(M)\oplus P(N)\oplus (S'(X)\oplus R(X))\oplus (\mathbb{Z}B_X\oplus\mathbb{Z}\Sigma_X),
\end{equation*}
where 
\begin{equation*}
(S'(X)\oplus R(X))=(\mathbb{Z}S_1\oplus\mathbb{Z}R_1)\oplus\dotsc\oplus(\mathbb{Z}S_d\oplus\mathbb{Z}R_d).
\end{equation*}
The direct sums are all orthogonal, except the direct sums inside the brackets. In this decomposition of $H^2(X;\mathbb{Z})$, the restriction of the intersection form $Q_X$ to $P(M)$ and $P(N)$ is equal to the intersection form induced from $M$ and $N$ and has the structure
\begin{equation*}
\left(\begin{array}{cc} B_M^2+B_N^2 &1 \\ 1& 0 \\ \end{array}\right)
\end{equation*} 
on $\mathbb{Z}B_X\oplus\mathbb{Z}\Sigma_X$ and the structure   
\begin{equation*}
\left(\begin{array}{cc} S_i^2 &1 \\ 1& 0 \\ \end{array}\right)
\end{equation*}  
on each summand $\mathbb{Z}S_i\oplus\mathbb{Z}R_i$.
\end{thm}
\begin{defn} We call a basis for $H^2(X;\mathbb{Z})$ with these properties a {\em normal form basis} or {\em normal form decomposition}.
\end{defn}
A normal form basis is in general not unique. However, it will follow from Lemma \ref{rim tori depend only on alpha} that a basis for $R(X)$ is determined by a basis for $\text{ker}(i_M\oplus i_N)$, where 
\begin{equation*}
i_M\oplus i_N\colon H_1(\Sigma)\rightarrow H_1(M)\oplus H_1(N).
\end{equation*}
Fixing such a basis determines a basis of rim tori in each fibre sum as the gluing diffeomorphism changes. The choice of the basis for $S(X)$ is then determined by the basis for $R(X)$ in the following way:
\begin{lem}\label{lem S1,...,Sd unique} Let $S_1',\dotsc,S_d'$ be elements in $H^2(X;\mathbb{Z})$ such that 
\begin{equation*}
H^2(X;\mathbb{Z})=P(M)\oplus P(N)\oplus (S'(X)\oplus R(X))\oplus (\mathbb{Z}B_X\oplus\mathbb{Z}\Sigma_X),
\end{equation*}
with 
\begin{equation*}
(S'(X)\oplus R(X))=(\mathbb{Z}S_1'\oplus\mathbb{Z}R_1)\oplus\dotsc\oplus(\mathbb{Z}S_d'\oplus\mathbb{Z}R_d)
\end{equation*}
is another normal form decomposition. Then $S_i'$ differs from $S_i$ only by a rim torus for all indices $i$.
\end{lem}
\begin{proof}
Since the restriction of the intersection form to $P(M)$, $P(N)$ and $\mathbb{Z}B_X\oplus\mathbb{Z}\Sigma_X$ is non-degenerate it follows that
\begin{equation*}
S_i'=\sum_i\alpha_{ij}S_j+\sum_i\beta_{ij}R_j.
\end{equation*}
Since $S_i'\cdot R_k=\delta_{ik}$ it follows that $\alpha_{ij}=\delta_{ij}$. This implies
\begin{equation*}
S_i'=S_i+\sum_i\beta_{ij}R_j,
\end{equation*}
hence the claim.
\end{proof}
\begin{rem}
Note that the self-intersections of the vanishing classes can change if we add a rim torus. Therefore we cannot make an {\em a priori} statement about the numbers $S_i^2$. In fact, we can add to $S_i^2$ any even number in this way. In particular, whether the intersection form is even or odd does not depend on the choice of the surfaces $S_i$, as it should be.
\end{rem}
The normal form decomposition for the cohomology of $X$ given by Theorem \ref{formula H^2 no torsion}
\begin{equation*}
H^2(X)=P(M)\oplus P(N)\oplus (S'(X)\oplus R(X))\oplus (\mathbb{Z}B_X\oplus\mathbb{Z}\Sigma_X)
\end{equation*}
can be compared to the splitting
\begin{equation}\label{decomp of H^2 in A,B,P II}
H^2(M)=P(M)\oplus\mathbb{Z}B_M\oplus \mathbb{Z}\Sigma_M.
\end{equation}
from equation \eqref{decomp of H^2 in A,B,P}. In particular, we can decompose the perpendicular classes of $X$ as
\begin{equation*}
P(X)=P(M)\oplus P(N)\oplus (S'(X)\oplus R(X)).
\end{equation*}
\begin{cor}\label{general case embedding HM HN in HX}
Under the assumptions of Theorem \ref{formula H^2 no torsion} there exists a group monomorphism $H^2(M)\rightarrow H^2(X)$ given by mapping
\begin{equation*}
P(M)\stackrel{\text{Id}}{\rightarrow}P(M), B_M\mapsto B_X,\Sigma_M\mapsto\Sigma_X.
\end{equation*}
There exists a similar monomorphism $H^2(N)\rightarrow H^2(X)$ given by
\begin{equation*}
P(N)\stackrel{\text{Id}}{\rightarrow}P(N), B_N\mapsto B_X,\Sigma_N\mapsto\Sigma_X',
\end{equation*}
where $\Sigma_X'=\Sigma_X+R_C$. The cohomology groups $H^2(M)$ and $H^2(N)$ can be realized as direct summands of $H^2(X)$. In general, the embeddings do not preserve the intersection form, because $B_X^2=B_M^2+B_N^2$. The images of both embeddings have non-trivial intersection and in general do not span $H^2(X)$ because of the rim tori and vanishing classes. 
\end{cor}
Theorem \ref{formula H^2 no torsion} is proved by constructing surfaces representing the basis $S_1,\dotsc,S_d$ of vanishing classes. The construction of these basis elements is rather lengthy and will be done step by step. First choose a basis $\alpha_1,\dotsc,\alpha_d$ for the free abelian group 
\begin{equation*}
\text{ker}(i_M\oplus i_N\colon H_1(\Sigma)\rightarrow H_1(M)\oplus H_1(N)).
\end{equation*}
\begin{lem}\label{ker iMiN direct summand} The subgroup $\text{ker}(i_M\oplus i_N)$ is a direct summand of $H_1(\Sigma)$.
\end{lem}
\begin{proof} Suppose that $\alpha\in \text{ker}(i_M\oplus i_N)$ is divisible by an integer $c>1$ so that $\alpha=c\alpha'$ with $\alpha'\in H_1(\Sigma)$. Then $ci_M\alpha'=0=ci_N\alpha'$. Since $H_1(M)$ and $H_1(N)$ are torsion free this implies that $\alpha'\in \text{ker}(i_M\oplus i_N)$. Hence $\text{ker}(i_M\oplus i_N)$ is a direct summand by the following fundamental lemma on free abelian groups (for a proof see e.g.~\cite[Lemma 6.15]{BC}).
\end{proof}
\begin{lem}\label{fund lemma free abelian} Let $G$ be a finitely generated free abelian group and $H\subset G$ a subgroup. Then there exists a basis $e_1,\ldots,e_n$ of $G$ such that $H$ is generated by the elements $a_1e_1,\ldots,a_ne_n$ for certain integers $a_1,\ldots,a_n$.
\end{lem}
It follows that the homology classes $\alpha_i$ are indivisible, hence we can represent them by embedded closed curves in $\Sigma$, see \cite{Mey}. We will often use the following lemma.
\begin{lem}
Let $\gamma$ be a possibly disconnected embedded closed curve on $M'$ such that $\mu_M\gamma$ is null-homologous in $M'$. Then $\gamma$ bounds an embedded oriented surface in $M'$ that is transverse to the boundary $\partial M'$.
\end{lem}
This surface can be thought of as a spanning surface in the 4-manifold $M'$ for the curve $\gamma$. Let $B_M'$ and $B_N'$ denote the surfaces in $M'$ and $N'$ bounding the meridians, obtained from $B_M$ and $B_N$ by deleting a disk. The class $B_X$ is sewed together from these surfaces along the meridians in $\partial M'$ and $\partial N'$. We also want to realize the homology classes $S_i$ by embedded surfaces in $X$. In general, surfaces that bound curves on the boundary $\partial M'$ are always oriented as follows:
\begin{defn}\label{defn orientation of surfaces bounding M'} Suppose that $L_M$ is a connected, orientable surface in $M'$ which is transverse to the boundary $\partial M'$ and bounds an oriented curve $c_M$ on $\partial M'$. In a collar of the boundary of the form $\Sigma\times S^1\times I$, we can assume that the surface $L_M$ has the form $c_M\times I$. The orientation of $L_M$ is then chosen such that it induces on $c_M\times I$ the orientation of $c_M$ followed by the orientation of the intervall $I$ pointing out of $M'$.
\end{defn}
This definition applies in particular to the surfaces $B_M'$ and $B_N'$ bounding the meridians in $M'$ and $N'$: The surfaces $\Sigma_M$ and $\Sigma_N$ are oriented by the embeddings $i_M,i_N$ from a fixed oriented surface $\Sigma$. The surfaces $B_M$ and $B_N$ are oriented such that $\Sigma_MB_M=+1$ and $\Sigma_NB_N=+1$. 
\begin{lem} The restriction of the orientation of $B_M$ to the punctured surface $B_M'$ is equal to the orientation determined by Definition \ref{defn orientation of surfaces bounding M'}.
\end{lem}
\begin{proof} On the 2-disk $D$ in $B_M$, obtained by the intersection with the tubular neigbourhood $\nu\Sigma_M$, the restriction of the orientation is given by the direction pointing into $M'$ followed by the orientation of the meridian $\sigma^M$. This is equal to the orientation of $\sigma^M$ followed by the direction pointing out of $M'$.
\end{proof}

The extension $\Phi$ of the gluing diffeomorphism $\phi$ (see equation \eqref{eq gluing diff Phi}) inverts on the 2-disk $D$ the inside-outside direction and the direction along the boundary $\partial D$. Hence the oriented surfaces $B_M'$ and $B_N'$ sew together to define an oriented surface $B_X$ in $X$. 

\subsection{Construction of the surfaces $D_i^M$}

\begin{defn} For the basis $\alpha_1,\dotsc,\alpha_d$ of $\text{ker}(i_M\oplus i_N)$ let $\alpha_i^M$ and $\alpha_i^N$ denote embedded curves on the push-offs $\Sigma^M$ and $\Sigma^N$ in $\partial\nu\Sigma_M$ and $\partial\nu\Sigma_N$ which are parallel to the curves $i_M\alpha_i$ and $i_N\alpha_i$ on $\Sigma_M$ and $\Sigma_N$ under the given trivialization of the tubular neighbourhoods. The curves $\alpha_i^M$ and $\alpha_i^N$ are null-homologous in $M'$ and $N'$ by Lemma \ref{parallel curve adapted}. We fix oriented embedded surfaces $D_i^M$ in $M'$ transverse to the boundary $\partial M'$ and bounding the curves $\alpha_i^M$. Similarly we fix embedded surfaces $D_i^N$ in $N'$ bounding the curves $\alpha_i^N$.
\end{defn} 
We want to determine the intersection numbers of the surfaces $D_i^M$ with the rim tori in $M'$. Let $\gamma_1,\dotsc,\gamma_{2g}$ denote the basis of $H_1(\Sigma)$. We defined a basis for $H_2(\Sigma\times S^1)$ given by the elements $\Gamma_i=PD(\gamma_i^*)$. On the 3-manifold $\Sigma\times S^1$ we have 
\begin{equation*}
\Gamma_i\cdot\gamma_j=\delta_{ij}.
\end{equation*}
More generally, suppose that a class $T$ represents the element $\sum_{i=1}^{2g}c_i\gamma_i^*$. Then $PD(T)=\sum_{i=1}^{2g}c_i\Gamma_i$ and
\begin{equation*}
PD(T)\cdot \gamma_j=\langle T, \gamma_j\rangle = c_j.
\end{equation*}
These relations also hold on $\partial M'$ and $\partial N'$. Consider the closed, oriented curve $\gamma_j$ on $\Sigma$. We view $\gamma_j$ as a curve on the push-off $\Sigma^M$ in $M'$. It defines a small annulus $\gamma_j\times I$ in a collar of the form $\Sigma\times S^1\times I$ of the boundary $\partial M'$. On the annulus we choose the orientation given by Definition \ref{defn orientation of surfaces bounding M'}. Then the intersection number of $\Gamma_i^M$ and $\gamma_j\times I$ in the manifold $M'$ is given by
\begin{equation*}
\Gamma_i^M\cdot (\gamma_j\times I)=\delta_{ij}
\end{equation*}
according to the orientation convention for $M'$, cf.~Section \ref{Sect Gompf sum general}.

More generally, suppose that 
\begin{equation*}
e=\sum_{i=1}^{2g}e_i\gamma_i
\end{equation*}
is an oriented curve on $\Sigma$ and $E_M$ the annulus $E_M=e\times I$ defined by $e$. Let $R_T^M$ be a rim torus in $M'$ induced from an element $T\in H^1(\Sigma)$. Then $R_T^M$ is the image of
\begin{equation*}
\sum_{j=1}^{2g}\langle T, \gamma_j\rangle \Gamma_j^M
\end{equation*}
under the inclusion of $\partial M'$ in $M'$. We then have
\begin{align*}
R_T^M\cdot E_M &= \sum_{i=1}^{2g}\langle T, \gamma_j\rangle e_j\\\
&= \langle T, e\rangle.
\end{align*}
\begin{lem}\label{lem intersect rim TM ann EM} With our orientation conventions, the algebraic intersection number of a rim torus $R_T^M$ and an annulus $E_M$ as above is given by $R_T^M\cdot E_M=\langle T, e\rangle$.
\end{lem} 
In particular we get:
\begin{lem}\label{lem intersect rim TM DiM} Let $R_T^M$ be the rim torus associated to an element $T$ in $H^1(\Sigma)$. Then $R_T^M\cdot D_i^M=\langle T, \alpha_i\rangle$.
\end{lem}

\subsection{Construction of the surfaces $U_i^N$}

Let $V_i^N$ denote the embedded surface in the closed manifold $N$, obtained by smoothing the intersection points of $x_N(\alpha_i)$ copies of $B_N$ and $x_N(\alpha_i)B_N^2$ disjoint copies of the push-off $\Sigma^N$ with the opposite orientation. Here
\begin{equation*}
x_N(\alpha_i)=\langle C,\alpha_i\rangle
\end{equation*}
as before. The surface $V_i^N$ represents the class $x_N(\alpha_i)B_N-(x_N(\alpha_i)B_N^2)\Sigma_N$ and has zero intersection with the class $B_N$ and the surfaces in $P(N)$. We can assume that the copies of $\Sigma^N$ are contained in the complement of the tubular neighbourhood $\nu\Sigma_N$. Deleting the interior of the tubular neigbourhood we obtain a surface $U_i^N$ in $N'$ bounding the disjoint union of $x_N(\alpha_i)$ parallel copies of the meridian $\sigma^N$ on the boundary $\partial N'$. The surface has zero intersection with the surface $B_N'$ and the rim tori in $R(N')$. In addition, the intersection number with the surface $\Sigma^N$ is given by $U_i^N\cdot \Sigma^N=x_N(\alpha_i)$.

\subsection{Definition of vanishing surfaces $S_i$}\label{subsection defn vanishing surfaces Si}

Suppose that the gluing diffeomorphism $\phi$ is isotopic to the trivial diffeomorphism that identifies the push-offs, or equivalently the cohomology class $C$ equals zero. We can then assume without loss of generality that $\phi$ is equal to the trivial diffeomorphism, hence the curves $\alpha_i^M$ and $\alpha_i^N$ on the boundaries $\partial M'$ and $\partial N'$ get identified by the gluing map. We set
\begin{align*}
S_i^M&=D_i^M\\
S_i^N&=D_i^N
\end{align*}
and define the vanishing class $S_i$ by sewing together the surfaces $S_i^M$ and $S_i^N$.

In the case that $C$ is different from zero the curves $\phi\circ\alpha_i^M$ and $\alpha_i^N$ do not coincide. By Lemma \ref{act phi hom 1} we know that the curve $\phi\circ\alpha_i^M$ is homologous to
\begin{equation*}
\alpha_i^N+\langle C,\alpha_i \rangle\sigma^N=\alpha_i^N+x_N(\alpha_i)\sigma^N.
\end{equation*}
Let $\nu\bar{\Sigma}_N$ be a tubular neighbourhood of slightly larger radius than $\nu\Sigma_N$ (we indicate curves and surfaces bounding curves on $\nu\bar{\Sigma}_N$ by a bar). On the boundary of $\nu\bar{\Sigma}_N$ we consider the curve $\bar{\alpha}_i^N$ and $x_N(\alpha_i)$ parallel copies of the meridian $\bar{\sigma}^N$, disjoint from $\bar{\alpha}_i^N$. The curve $\bar{\alpha}_i^N$ bounds a surface $\bar{D}_i^N$ as before and the parallel copies of the meridian bound the surface $\bar{U}_i^N$. We can connect the curve $\phi\circ\alpha_i^M$ on $\partial\nu\Sigma_N$ to the disjoint union of the meridians and the curve $\bar{\alpha}_i^N$ on $\partial\nu\bar{\Sigma}_N$ by an embedded connected oriented surface $Q_i^N$ realizing the homology. 
\begin{lem}\label{lem QiN disjoint from annulus}  We can assume that the homology $Q_i^N$ is disjoint from an annulus of the form $\{*\}\times S^1\times I$ in the submanifold $\Sigma_N\times S^1\times I$ bounded by $\partial \nu\Sigma_N$ and $\partial\nu\bar{\Sigma}_N$.
\end{lem}
\begin{proof} This follows because the isomorphism $H_1(\Sigma^*)\rightarrow H_1(\Sigma)$ induces an isomorphism $H_1(\Sigma^*\times S^1)\rightarrow H_1(\Sigma\times S^1)$, where $\Sigma^*$ denotes the punctured surface $\Sigma\setminus\{*\}$.
\end{proof}
We set
\begin{align*}
S_i^M&=D_i^M\\
S_i^N&=Q_i^N\cup \bar{U}_i^N\cup \bar{D}_i^N.
\end{align*}
The connected surfaces $S_i^M$ and $S_i^N$ sew together to define the vanishing surfaces $S_i$. The orientation of $S_i$ is defined as follows: The surfaces $S_i^M$ and $S_i^N$ are oriented according to Definition \ref{defn orientation of surfaces bounding M'}. The orientation of $I$ is reversed by the extension $\Phi$ of the gluing diffeomorphism, while the orientations of the curves $\alpha_i^M$ and $\phi\circ\alpha_i^M$ are identified. This implies that the surface $S_i^M$ with its given orientation and the surface $S_i^N$ with the {\em opposite} orientation sew together to define an oriented surface $S_i$ in $X$. 

\begin{prop} For arbitrary gluing diffeomorphism $\phi$, the vanishing classes $S_i$ satisfy $S_i\cdot \Sigma_X=0$. We also have $B_X\cdot\Sigma_X=1$ and $B_X^2=B_M^2+B_N^2$.
\end{prop}\begin{proof} We only have to check the first equation. This can be done on either the $M$ or the $N$ side: On the $M$ side we have we can assume that $\Sigma^M$ is given by the push-off $\Sigma_M\times\{p\}$ in the boundary $\Sigma_M\times S^1$ and that the curves $\alpha_i^M$ are embedded in a different push-off $\Sigma_M\times\{q\}$ with $p\neq q$. This implies that $D_i^M$ is disjoint from $\Sigma^M$, hence
\begin{equation*}
S_i\cdot \Sigma_X=D_i^M\cdot \Sigma^M=0.
\end{equation*}
On the $N$ side we have
\begin{align*}
S_i\cdot\Sigma_X&=S_i\cdot(\Sigma_X'-R_C)=-S_i^N\cdot(\Sigma_N-R_C^N)\\
&=-U_i^N\cdot\Sigma_N+D_i^N\cdot R_C^N =-x_N(\alpha_i)+\langle C,\alpha_i\rangle=0.
\end{align*}
\end{proof}
The rim torus $R_T^M$ associated to an element $T$ in $H^1(\Sigma)$ induces under the inclusion $M'\rightarrow X$ a rim torus in $X$, denoted by $R_T$. It has zero intersection with the surfaces $\Sigma_X$, $B_X$ and any other rim torus. The intersection with a vanishing class $S_i$ is given by $R_T\cdot S_i=\langle T,\alpha_i\rangle$ according to Lemma \ref{lem intersect rim TM ann EM}.

\subsection{Normal form basis for the intersection form}
In the construction of the vanishing classes $S_i$ from the surfaces $S_i^M$ and $S_i^N$ we can assume that the curves $\alpha_i^M$ are contained in pairwise disjoint parallel copies of the push-off $\Sigma^M$ in $\partial\nu\Sigma^M$ of the form $\Sigma_M\times\{p_i\}$. Hence the vanishing classes $S_i$ in $X$ do not intersect on $\partial M'=\partial N'$ and we can assume that they have only transverse intersections. 

We now simplify the intersection form on $S'(X)\oplus R(X)$. According Lemma \ref{ker iMiN direct summand} we can complete the basis $\alpha_1,\dotsc,\alpha_d$ for $\text{ker}\,(i_M\oplus i_N)$ by certain elements $\beta_{d+1},\dotsc,\beta_{2g}\in H_1(\Sigma)$ to a basis of $H_1(\Sigma)$. Let 
\begin{equation*}
\alpha_1^*,\dotsc,\alpha_d^*,\beta_{d+1}^*,\dotsc,\beta_{2g}^*
\end{equation*}
denote the dual basis of $H^1(\Sigma)$ and $R_1,\dotsc,R_{2g}$ the corresponding rim tori in $H^2(X)$. Then
\begin{align*}
S_i\cdot R_j&= \delta_{ij},\quad\text{for $1\leq j\leq d$}\\
S_i\cdot R_j&= 0,\quad\text{for $d+1\leq j\leq 2g$}.
\end{align*}    
This implies that the elements $R_1,\dotsc,R_d$ are a basis of $R(X)$ and $R_{d+1},\dotsc,R_{2g}$ are null-homologous, since the cohomology of $X$ is torsion free. The following is easy to check:
\begin{lem}\label{rim tori depend only on alpha} If we make a different choice for the basis elements $\beta_{d+1},\dotsc,\beta_{2g}$, then $R_1,\dotsc,R_d$ only change by null-homologous rim tori. 
\end{lem}

The surfaces $S_i$ are simplified as follows: Let $r_{ij}=S_i\cdot S_j$ for $i,j=1,\dotsc,d$ denote the matrix of intersection numbers and let 
\begin{equation*}
S_i'=S_i-\sum_{k>i}r_{ik}R_k.
\end{equation*}
The surfaces $S_i'$ are tubed together from the surfaces $S_i$ and certain rim tori. They can still be considered as vanishing classes sewed together from surfaces in $M'$ and $N'$. We then have $S_i'\cdot S_j'=0$ for $i\neq j$. Denote these new vanishing classes again by $S_1,\dotsc,S_d$ and the subgroup spanned by them in $S'(X)\oplus R(X)$ again by $S'(X)$. The intersection form on $S'(X)\oplus R(X)$ now has the form as in Theorem \ref{formula H^2 no torsion}.

\begin{rem}\label{rem on construction of Ri Si}
We can choose the basis $\gamma_1,\dotsc,\gamma_{2g}$ of $H_1(\Sigma)$ we started with in Section \ref{subsect action phi} as
\begin{align*}
\gamma_i&=\alpha_i, \quad \text{for $1\leq i\leq d$} \\
\gamma_i&=\beta_i, \quad \text{for $d+1\leq i\leq 2g$.}
\end{align*}
This choice does not depend on the choice of $C$ since $\alpha_1,\dotsc,\alpha_d$ are merely a basis for $\text{ker}(i_M\oplus i_N)$. Then the images of the classes $\Gamma_1^M,\dotsc,\Gamma_d^M$ under the inclusion $\partial M'\rightarrow M'\rightarrow X$ are equal to the rim tori $R_1,\dotsc,R_d$ and the rim tori determined by $\Gamma_{d+1}^M,\dotsc,\Gamma_{2g}^M$ are null-homologous in $X$. In this basis the rim torus $R_C$ in $X$ is given by
\begin{equation*}
R_C=-\sum_{i=1}^da_iR_i,
\end{equation*}   
where $a_i=\langle C, \alpha_i\rangle$. Having chosen the basis for $\text{ker}(i_M\oplus i_N)$ we have a canonical choice of the rim tori basis in each fibre sum. By Lemma \ref{lem S1,...,Sd unique} this determines a basis for the vanishing surfaces up to rim tori summands.
\end{rem}  
Let $P_1,\ldots,P_n$ denote a basis for $P(M)$, consisting of embedded surfaces in $M$ disjoint from $\Sigma_M$ and $B_M$. We can consider these classes as surfaces in $M'$ and hence in $X$. In general, they will have intersections with the vanishing surfaces but they are disjoint from the rim tori which are supported near $\Sigma_M$. Adding to each of the $P_i$ a suitable linear combination or rim tori we can achieve that these new classes have zero intersection with the vanishing surfaces. Mapping the chosen basis of $P(M)$ to the corresponding surfaces plus rim tori we get an embedding of $P(M)$ into $H^2(X;\mathbb{Z})$ that does not depend on the choice of gluing and does not change the intersection form on $P(M)$. A similar construction works for $P(N)$. Finally, we can add to the surface $B_X$ certain rim tori so that the new class has zero intersection with the vanishing surfaces. The basis for the second cohomology of $X$ is now in normal form.

\section{A formula for the canonical class}\label{formula KX calculation}

The canonical class $K_X$ of a symplectic 4-manifold $(X,\omega_X)$ is defined as minus the first Chern class. In this section we derive a formula for the canonical class of the symplectic generalized fibre sum $X$ of two symplectic 4-manifolds $M$ and $N$ along embedded symplectic surfaces $\Sigma_M$ and $\Sigma_N$ of genus $g$. We make the same assumptions as in the previous section, i.e.~that $\Sigma_M$ and $\Sigma_N$ represent indivisible classes of square zero and the cohomologies of $M$, $N$ and $X$ are torsion free.

The idea of the proof is to choose a normal form basis for $H^2(X)$ as in Theorem \ref{formula H^2 no torsion} and decompose the canonical class $K_X$ as
\begin{equation}\label{eq formal decomp canonical KX}
K_X=\kappa_M+\kappa_N+\sum_{i=1}^ds_iS_i+\sum_{i=1}^dr_iR_i+b_XB_X+\sigma_X\Sigma_X,
\end{equation}
where $\kappa_M\in P(M)$ and $\kappa_N\in P(N)$. Since the cohomology of $X$ is assumed to be torsion free, the coefficients in equation \eqref{eq formal decomp canonical KX} can be determined from the intersection numbers of the canonical class $K_X$ and the basis elements of $H^2(X)$.

\subsection{Definition of the symplectic form on $X$}
We first recall the definition of the symplectic generalized fibre sum by the construction of Gompf \cite{Go}. Let $(M,\omega_M)$ and $(N,\omega_N)$ be closed, symplectic 4-manifolds and $\Sigma_M,\Sigma_N$ embedded symplectic surfaces of genus $g$ and self-intersection zero. The symplectic generalized fibre sum is constructed using the following lemma. First we make the following definition:
\begin{defn}
Recall that we have fixed trivializations of tubular neighbourhoods $\nu\Sigma_M$ and $\nu\Sigma_N$ determined by the framings $\tau_M$ and $\tau_N$. Let $\phi\colon \partial\nu\Sigma_M\rightarrow\partial\nu\Sigma_N$ be the gluing diffeomorphism. We also choose an auxiliary framing $\tau_N'$ such that in the framings $\tau_M$ and $\tau_N'$ the gluing diffeomorphism is given by the trivial diffeomorphism 
\begin{equation*}
\Sigma\times S^1\rightarrow \Sigma\times S^1, (z,\alpha)\mapsto(z,\overline{\alpha})
\end{equation*}
that identifies the push-offs.
\end{defn}

We can identify the interior of the tubular neighbourhoods with $\Sigma\times D$, where $D$ denotes the open disk of radius $1$ in $\mathbb{R}^2$. 

\begin{lem}\label{lem Gompf symp form} The symplectic structures $\omega_M$ and $\omega_N$ can be deformed by rescaling and isotopies such that both restrict on the tubular neighbourhoods $\nu\Sigma_M$ and $\nu\Sigma_N$ with respect to the framings $\tau_M$ and $\tau_N'$ to the same symplectic form
\begin{equation*}
\omega=\omega_\Sigma+\omega_D,
\end{equation*}
where $\omega_D$ is the standard symplectic structure $\omega_D=dx\wedge dy$ on the open unit disk $D$ and $\omega_\Sigma$ is a symplectic form on $\Sigma$. 
\end{lem}
\begin{proof} We use Lemma 1.5 and Lemma 2.1 from \cite{Go}. Choose an arbitrary symplectic form $\omega$ on $\Sigma$ and rescale $\omega_M$ and $\omega_N$ such that 
\begin{equation*}
\int_{\Sigma_M}\omega_M=\int_{\Sigma_N}\omega_M=\int_{\Sigma}\omega.
\end{equation*}
We can then isotope the embeddings $i_M\colon \Sigma\rightarrow M$ and $i_N\colon \Sigma\rightarrow N$ without changing the images, such that both become symplectomorphisms onto $\Sigma_M$ and $\Sigma_N$. The isotopies can be realized by taking fixed embeddings $i_M, i_N$ and composing them with isotopies of self-diffeomorphisms of $M$ and $N$ (because $M$ and $N$ are closed manifolds). Hence we can consider the embeddings to be fixed and instead change the symplectic forms $\omega_M$ and $\omega_N$ by pulling them back under isotopies of self-diffeomorphisms. 

The embeddings $\tau_M\colon \Sigma\times D\rightarrow M$ and $\tau_N'\colon \Sigma\times D\rightarrow N$ are symplectic on the submanifold $\Sigma\times 0$. We can isotope both embeddings to new embeddings which are symplectic on small neighbourhoods of $\Sigma\times 0$ with respect to the symplectic form $\omega+\omega_D$ on $\Sigma\times D$. Since $\Sigma$ is compact, we can assume that both are symplectic on $\Sigma\times D_{\epsilon}$ where $D_{\epsilon}$ denotes the disk with radius $\epsilon<1$. Again the isotopies can be achieved by considering $\tau_M$ and $\tau_N'$ unchanged and pulling back the symplectic forms on $M$ and $N$ under isotopies of self-diffeomorphisms. 

It is easier to work with disks of radius $1$: We rescale the symplectic forms $\omega_M, \omega_N$ and $\omega+\omega_D$ by the factor $1/\epsilon^2$. Then we compose the symplectic embeddings $\tau_M$ and $\tau_N'$ on $(\Sigma\times D_{\epsilon}, (1/\epsilon^2)(\omega+\omega_D))$ with the symplectomorphism 
\begin{align*}
\Sigma\times D&\rightarrow \Sigma\times D_{\epsilon}\\
(p,(x,y))&\mapsto (p, (\epsilon x,\epsilon y),
\end{align*} 
where $\Sigma\times D$ has the symplectic form $(1/\epsilon^2)\omega+\omega_D$. We then define $\omega_\Sigma=(1/\epsilon^2)\omega$ to get the statement we want to prove.
\end{proof}

It is useful to introduce polar coordinates $(r,\theta)$ on $D$ such that $x=r\cos\theta$ and $y=r\sin\theta$, hence
\begin{align*}
dx&=dr\cos\theta-r\sin\theta d\theta\\
dy&=dr\sin\theta + r\cos\theta d\theta.
\end{align*}
Then $\omega_D=rdr\wedge d\theta$. To form the generalized fibre sum, the manifolds $M\setminus\Sigma_M$ and $N\setminus \Sigma_N$ are glued together along the collars $\mbox{int}\,\nu\Sigma_M\setminus \Sigma_M$ and $\mbox{int}\,\nu\Sigma_N\setminus \Sigma_N$ of radius $1$ by an orientation and $S^1$-fibre preserving diffeomorphism that in the framings $\tau_M$ and $\tau_N'$ is given by
\begin{equation}\label{definition capital Phi glue symp}
\begin{split}
\Phi\colon (D\setminus \{0\})\times \Sigma&\rightarrow  (D\setminus\{0\})\times \Sigma\\
(r,\theta, z)&\mapsto (\sqrt{1-r^2}, -\theta, z).
\end{split}
\end{equation}
The action of $\Phi$ on the 1-forms $dr$ and $d\theta$ is given by
\begin{align*}
\Phi^*dr&= d(r\circ\Phi)=d\sqrt{1-r^2}=\frac{-r}{\sqrt{1-r^2}}dr\\
\Phi^*d\theta&= d(\theta\circ \Phi)=-d\theta.
\end{align*}
This implies that
\begin{align*}
\Phi^*\omega_D&=\omega_D\\
\Phi^*\omega_\Sigma&=\omega_\Sigma.
\end{align*}
Hence
\begin{equation*}
\Phi^*(\omega_D+\omega_\Sigma)=\omega_D+\omega_\Sigma 
\end{equation*}
and the symplectic forms on $M'$ and $N'$ given by Lemma \ref{lem Gompf symp form} glue together to define a symplectic form on the fibre sum $X$. This is the symplectic form according to the Gompf construction.

\begin{rem} The Gompf construction for the symplectic generalized fibre sum can only be done if (after a rescaling) the symplectic structures $\omega_M$ and $\omega_N$ have the same volume on $\Sigma_M$ and $\Sigma_N$:
\begin{equation*}
\int_{\Sigma_M}\omega_M=\int_{\Sigma_N}\omega_N.
\end{equation*}
To calculate this number, both $\Sigma_M$ and $\Sigma_N$ have to be oriented which we have assumed {\em a priori}. It is not necessary that this number is positive, the construction can also be done for negative volume. In the standard case the orientation induced by the symplectic forms coincides with the given orientation on $\Sigma_M$ and $\Sigma_N$ and is the opposite orientation in the second case.  
\end{rem}
We split the canonical class $K_X$ as in equation \eqref{eq formal decomp canonical KX}
\begin{equation*}
K_X=\kappa_M+\kappa_N+\sum_{i=1}^ds_iS_i+\sum_{i=1}^dr_iR_i+b_XB_X+\sigma_X\Sigma_X.
\end{equation*}
The coefficients in this formula can be determined using intersection numbers. We assume that $\Sigma_M$ and $\Sigma_N$ are oriented by the symplectic forms $\omega_M$ and $\omega_N$. Then $\Sigma_X$ is a symplectic surface in $X$ of genus $g$ and self-intersection $0$, oriented by the symplectic form $\omega_X$. This implies by the adjunction formula
\begin{equation*}
b_X=K_X\Sigma_X=2g-2,
\end{equation*}
hence
\begin{equation*}
\sigma_X=K_XB_X-(2g-2)(B_M^2+B_N^2).
\end{equation*}
Every rim torus $R_j$ is a linear combination of embedded Lagrangian tori of self-intersection zero in $X$. Since the adjunction formula holds for each one of them,
\begin{equation*}
s_j=0,\quad\mbox{for all $j=1,\dotsc,d$,}
\end{equation*}
hence also
\begin{equation*}
r_j=K_XS_j.
\end{equation*}
To determine the coefficient $\kappa_M$ we know that $\eta_M^*K_X=K_{M'}=\rho_M^*K_M$. This implies that the intersection of a class in $P(M)$ with $K_X$ is equal to its intersection with $K_M$. We can decompose $K_M$ as
\begin{equation}\label{direct sum decomp KM, BM, SigmaM}
K_M=\overline{K_M}+(K_MB_M-(2g-2)B_M^2)\Sigma_M+(2g-2)B_M,
\end{equation}
where the element $\overline{K_M}$, defined by this equation, is in $P(M)$; see Lemma \ref{decomp alpha P(M)}. It is then clear that 
\begin{equation*}
\kappa_M=\overline{K_M}
\end{equation*}
under the chosen embedding of $P(M)$ into $H^2(X;\mathbb{Z})$. Similarly, $\kappa_N=\overline{K_N}$. It remains to determine the intersections of the canonical class with the vanishing classes $B_X$ and $S_1,\dotsc,S_d$. This is more difficult and will be done in the following subsections.

\subsection{Construction of a holomorphic 2-form on $X$}
In the construction of the symplectic form on $X$ we first chose symplectic forms on $M$ and $N$ which are standard on the tubular neighbourhoods given by $\tau_M$ and $\tau_N'$. Similarly, we want to choose compatible almost complex structures and holomorphic 2-forms on $M$ and $N$ which are standard on the tubular neighbourhoods and then understand how they behave under gluing. We will see that this is quite easy for the almost complex structure, but slightly more difficult for the holomorphic 2-form, which is a section of the canonical bundle. 

We first need compatible almost complex structures: We choose the standard almost complex structure $J_D$ on $D$ which maps $J_D\partial_x=\partial_y$ and $J_D\partial_y=-\partial_x$. In polar coordinates
\begin{align*}
J_D\partial_r&= \tfrac{1}{r}\partial_\theta\\
J_D\tfrac{1}{r}\partial_\theta&= -\partial_r.
\end{align*}
We also choose a compatible almost complex structure $J_\Sigma$ on $\Sigma$ which maps $J_\Sigma \partial_{z_1}=\partial_{z_2}$ and $J_\Sigma \partial_{z_2}=-\partial_{z_1}$ in suitable coordinates $z=(z_1,z_2)$ on the surface $\Sigma$. The almost complex structure $J_D+J_\Sigma$ on the tubular neighbourhood $D\times\Sigma$ in the framings $\tau_M$ and $\tau_N'$ extend to compatible almost complex structures on $M$ and $N$. This follows because compatible almost complex structures are the sections of a bundle with contractible fibres.

We want to construct an almost complex structure on $X$. Note that $\Phi$ maps the subset $S\times \Sigma\subset\nu\Sigma_M$, where $S$ is the circle of radius $\scriptstyle{\frac{1}{\sqrt{2}}}$, onto the same subset in $\nu\Sigma_N$. On $S\times\Sigma$ we have
\begin{align*}
\Phi_*\partial_r&=-\partial_r\\
\Phi_*\partial_\theta&=-\partial_\theta\\
\Phi_*\partial_{z_1}&=\partial_{z_1}\\
\Phi_*\partial_{z_2}&=\partial_{z_2}.
\end{align*}
Setting
\begin{equation*}
\Phi^*J=(\Phi_*)^{-1}\circ J\circ \Phi_*
\end{equation*}
where $J=J_D+J_\Sigma$ we obtain
\begin{align*}
(\Phi^*J)\partial_r&=\sqrt{2}\partial_\theta\\
(\Phi^*J)\partial_\theta&=-\tfrac{1}{\sqrt{2}}\partial_r\\
(\Phi^*J)\partial_{z_1}&=\partial_{z_2}\\
(\Phi^*J)\partial_{z_2}&=-\partial_{z_1}.
\end{align*}
Hence $\Phi^*J=J$ on $S\times \Sigma$ and the almost complex structures on $M'$ and $N'$ glue together to define a (continuous) compatible almost complex structure on $X$. As an aside note that we cannot have $\Phi^*J=J$ on the whole tubular neighbourhood because the gluing diffeomorphism $\Phi$ is not an isometry for the standard metric on the disk.

Recall that the sections of the canonical bundle $K_M$ are complex valued 2-forms on $M$ which are ``holomorphic'', i.e.~complex linear. We choose the holomorphic 1-form $\Omega_D=dx+idy$ on $D$, which can be written in polar coordinates as
\begin{equation}\label{def Omega hol}
\Omega_D=(dr+ird\theta)e^{i\theta}.
\end{equation}
This form satisfies $\Omega_D\circ J_D=i\Omega_D$. We also choose a holomorphic 1-form $\Omega_\Sigma$ on $\Sigma$. This form can be chosen such that it has precisely $2g-2$ different zeroes of index $+1$. We can assume that all zeroes are contained in a small disk $D_\Sigma$ around a point $q$ disjoint from the zeroes. The form $\Omega_D\wedge\Omega_\Sigma$ is then a holomorphic 2-form on $D\times\Sigma$ which has transverse zero set consisting of $2g-2$ parallel copies of $D$. This 2-form can be extended to holomorphic 2-forms on $M$ and $N$ as sections of the canonical bundles.   

These holomorphic 2-forms do not immediately glue together to define a holomorphic 2-form on $X$: On $S\times \Sigma$ we have
\begin{align*}
\Phi^*dr&= -dr\\
\Phi^*d\theta&=-d\theta,
\end{align*}
hence
\begin{equation*}
\Phi^*\Omega_D= -(dr+ird\theta)e^{-i\theta}=-\Omega_De^{-2i\theta}
\end{equation*}
and $\Omega_D$ is not invariant under the gluing diffeomorphism $\Phi$.

A continuous section $\Omega_X$ for the canonical line bundle $K_X$ can be constructed in the following way.
\begin{defn}
We call a holomorphic 2-form on $M$ or $N$ {\em standard} if on a tubular neighbourhood of $\Sigma_M$ or $\Sigma_N$ in the framings $\tau_M$ and $\tau_N'$ it is given by $\Omega_D\wedge \Omega_\Sigma$.
\end{defn}
Choose holomorphic 2-forms $\Omega_M$ and $-\Omega_N$ on $M$ and $N$ which are standard on tubular neighbourhoods of radius $1$. The manifold $X$ is constructed by gluing together the manifolds $M'$ and $N'$ obtained by deleting the interior of the tubular neighbourhoods of radius $\scriptstyle{\frac{1}{\sqrt{2}}}$. The holomorphic 2-form $\Omega_X$ is given on the $M$ and $N$ side of $X$ by $\Omega_M$ and $\Omega_N$ outside of the tubular neighbourhoods of radius $1$. We now have to define the form on both sides between radius $1$ and $\scriptstyle{\frac{1}{\sqrt{2}}}$. On the $N$ side we still choose the form $\Omega_N$ in this region. On the boundary $\partial N'=S\times \Sigma$ of radius $\scriptstyle{\frac{1}{\sqrt{2}}}$ we then have the holomorphic 2-form given by the restriction of $\Omega_N$, which is minus the standard form. It pulls back under the gluing diffeomorphism $\Phi$ to a holomorphic 2-form on $\partial M'=S\times \Sigma$. By the calculation above this form is given by
\begin{equation}\label{eq OmegaD^M wedge OmegaSigma e-2itheta+iC}
\Omega_M e^{-2i\theta}.
\end{equation}
Let $A'$ denote the annulus between radius $\scriptstyle{\frac{1}{\sqrt{2}}}$ and $1$. We want to change the form in equation \eqref{eq OmegaD^M wedge OmegaSigma e-2itheta+iC} over $A'\times \Sigma$ through a  holomorphic 2-form to the form $\Omega_M$ at radius $1$. This form will then be extended over the rest of $M$ using the 2-form $\Omega_M$. The change will be done by changing the function $e^{-2i\theta}$ at radius $\scriptstyle{\frac{1}{\sqrt{2}}}$ over $A'\times \Sigma$ to the constant function with value $1$ at radius $1$. This is not possible if we consider the functions as having image in $S^1$, because they represent different cohomology classes on $S^1\times \Sigma$. Hence we consider $S^1\subset\mathbb{C}$ and the change will involve crossings of zero.  We choose a smooth function $f\colon A'\times\Sigma\rightarrow \mathbb{C}$ which has $0$ as a regular value and satisfies on the boundaries
\begin{equation*}
f_{\scriptstyle{\frac{1}{\sqrt{2}}}}=e^{-2i\theta}\,\,\,\mbox{and}\,\,\,f_1\equiv 1.
\end{equation*}
The Poincar\'e dual of the zero set of $f$ is then the cohomology class of $S^1\times \Sigma$ determined by the $S^1$-valued function $e^{2i\theta}$. This class is $2\Sigma^M$ and is the obstruction to extending the function $f$ on the boundary into the interior. We get:
\begin{prop}\label{prop Omega' hol 2-form} There exists a 2-form $\Omega'$ on $A'\times\Sigma_M$ which is holomorphic for $J_D+J_\Sigma$ and satisfies:
\begin{itemize}
\item $\Omega'= \Omega_Me^{-2i\theta}$ at $r=\scriptstyle{\frac{1}{\sqrt{2}}}$ and $\Omega'= \Omega_M$ at $r=1$.
\item The zero set of the form $\Omega'$ represents the class $2\Sigma^M$ in the interior of $A'\times\Sigma_M$ and $2g-2$ parallel copies of $A'$. 
\end{itemize}
\end{prop}
\begin{cor}\label{cor holom Omega on X} There exists a symplectic form $\omega_X$ with compatible almost complex structure $J_X$ and holomorphic 2-form $\Omega_X$ on $X$ such that:
\begin{itemize}   
\item On the boundary $\partial\nu\Sigma_N$ of the tubular neighbourhood of $\Sigma_N$ in $N$ of radius $1$ the symplectic form and the almost complex structure are $\omega_X=\omega_D+\omega_\Sigma$ and $J_X=J_D+J_\Sigma$ while $\Omega_X$ is minus the standard form.
\item On the boundary $\partial\nu\Sigma_M$ of the tubular neighbourhood of $\Sigma_M$ in $M$ of radius $1$ the symplectic form and the almost complex structure are $\omega_X=\omega_D+\omega_\Sigma$ and $J_X=J_D+J_\Sigma$ while $\Omega_X$ is standard.
\item On the subset of $\nu\Sigma_N$ between radius $\scriptstyle{\frac{1}{\sqrt{2}}}$ and $1$, which is an annulus times $\Sigma_N$, the zero set of $\Omega_X$ consists of $2g-2$ parallel copies of the annulus.
\item On the subset of $\nu\Sigma_M$ between radius $\scriptstyle{\frac{1}{\sqrt{2}}}$ and $1$, which is an annulus times $\Sigma_M$, the zero set of $\Omega_X$ consists of $2g-2$ parallel copies of the annulus and a surface in the interior representing $2\Sigma^M$.
\end{itemize}
\end{cor}

\subsection{The formula for the canonical class of $X$}

We now calculate $K_XB_X$. 
\begin{lem}\label{lem comp K_XB_X} With the standard choice of orientation for $B_X$, we have $K_XB_X=K_MB_M+K_NB_N+2$. 
\end{lem}
\begin{proof} We extend the holomorphic 2-form $\Omega_D\wedge\Omega_\Sigma$ on the boundary $\partial\nu\Sigma_M$ of the tubular neighbourhood of $\Sigma_M$ in $M$ of radius $1$ to the holomorphic 2-form on $\nu\Sigma_M$ given by the same formula and then to a holomorphic 2-form on $M\setminus \mbox{int}\,\nu\Sigma_M$. The zero set of the resulting holomorphic 2-form $\Omega_M$ restricted to $\nu\Sigma_M=D_M\times \Sigma_M$ consists of $2g-2$ parallel copies of $D_M$. We can choose the surface $B_M$ such that it is parallel but disjoint from these copies of $D_M$ inside $\nu\Sigma_M$ and intersects the zero set of $\Omega_M$ outside transversely. The zero set on $B_M$ then consists of a set of points which count algebraically as $K_MB_M$. We can do a similar construction for $N$. We think of the surface $B_X$ as being glued together from the surfaces $B_M$ and $B_N$ by deleting in each a disk of radius $\scriptstyle{\frac{1}{\sqrt{2}}}$ in $D_M$ and $D_N$ around $0$. On the $M$ side we get two additional positive zeroes coming from the intersection with the class $2\Sigma^M$ in Corollary \ref{cor holom Omega on X} over the annulus in $D_M$ between radius $\scriptstyle{\frac{1}{\sqrt{2}}}$ and $1$. Adding these terms proves the claim.
\end{proof}      

It remains to calculate the intersections $K_X S_i$ which determine the rim tori contribution to the canonical class. We first define the meaning of intersection numbers of the canonical class $K_M$ pulled back to $M'$ and surfaces $L^M$ in $M'$ which bound curves on $\partial M'$.  
\begin{defn} Let $\Omega_\Sigma$ be a given 1-form on $\Sigma$ with $2g-2$ transverse zeroes, holomorphic with respect to a given almost complex structure $J_\Sigma$. Under the embedding $i_M$ and the trivialization $\tau_M$ of the normal bundle equip the tubular neighbourhood $\nu\Sigma_M$ with the almost complex structure $J_D+J_\Sigma$ and the holomorphic 2-form $\Omega_D\wedge \Omega_\Sigma$. Let $L^M$ be a compact oriented surface in $M'=M\setminus\mbox{int}\,\nu\Sigma_M$ which bounds a closed curve $c^M$ on $\partial\nu\Sigma_M$, disjoint from the zeroes of $\Omega_D\wedge \Omega_\Sigma$ on the boundary. Then $K_ML^M$ denotes the obstruction to extending the given section of $K_M$ on $c^M$ over the whole surface $L^M$. This is the number of zeroes one encounters when trying to extend the non-vanishing section of $K_M$ on $\partial L^M$ over all of $L^M$. 
\end{defn}
Recall that for the tubular neighbourhood of $\Sigma_N$ we had two different framings. There is an exactly analogous definition for $K_NL^N$ with almost complex structure $J_D+J_\Sigma$ and holomorphic 2-form $\Omega_D\wedge\Omega_\Sigma$ on the tubular neighbourhood $\nu\Sigma_N$ with respect to the framing $\tau_N$. If we choose the framing $\tau_N'$ we get a number that we denote by $K_N'L^N$. In general, $K_NL^N\neq K_N'L^N$ because the almost complex structures and the holomorphic 2-forms on the tubular neighbourhood  are different. We want to calculate the difference of these intersection numbers. We first have three lemmas: Recall that the framings are diffeomorphisms $\tau_N,\tau_N'\colon D^2\times \Sigma\rightarrow \nu\Sigma_N$. In the framing $\tau_N$ the gluing diffeomorphism $\phi$ is given by a map $C\colon \Sigma\rightarrow S^1$.
\begin{lem}
Let $\chi={\tau_N'}^{-1}\circ\tau_N$. Then
\begin{equation*}
\chi(r,\theta,z)=(r,-C(z)+\theta,z).
\end{equation*}
\end{lem}
\begin{proof}
Let $\Phi$ denote the gluing diffeomorphism on the tubular neighbourhoods minus the central surface. We have
\begin{equation*}
\tau_N^{-1}\circ\Phi\circ\tau_M(r,\theta,z)=(\sqrt{1-r^2},C(z)-\theta,z)
\end{equation*}
and
\begin{equation*}
{\tau_N'}^{-1}\circ\Phi\circ\tau_M(r,\theta,z)=(\sqrt{1-r^2},-\theta,z).
\end{equation*}
This implies the claim.
\end{proof}
Let $\Omega_N$ denote the 2-form on $\nu\Sigma_N$ given in the framing $\tau_N$ by $\Omega_D\wedge\Omega_\Sigma$. Similarly, let $\Omega'_N$ denote the 2-form on $\nu\Sigma_N$ given in the framing $\tau_N'$ by $\Omega_D\wedge \Omega_\Sigma$. By definition, this means that
\begin{align*}
\tau_N^*\Omega_N&=\Omega_D\wedge\Omega_\Sigma\\
{\tau_N'}^*\Omega_N'&=\Omega_D\wedge\Omega_\Sigma.
\end{align*}
Both are complex linear with respect to the corresponding almost complex structures $J_N$ and $J_N'$ which are standard on the tubular neighbourhood in the respective framings. We want to express $\Omega_N'$ in the framing $\tau_N$.
\begin{lem}\label{lem Omega_N' in tau_N}
We have
\begin{equation*}
\tau_N^*\Omega_N'=(dr+ird\theta-irdC)\wedge \Omega_\Sigma e^{i\theta-iC}.
\end{equation*}
\end{lem}
\begin{proof}
By definition,
\begin{align*}
\tau_N^*\Omega_N'&=\tau_N^*({\tau_N'}^{-1})^*\Omega_D\wedge\Omega_\Sigma\\
&=\chi^*\Omega_D\wedge\Omega_\Sigma.
\end{align*}
Since $\Omega_D=(dr+ird\theta)e^{i\theta}$ this implies the claim with the previous lemma.
\end{proof}
We also calculate the almost complex structure $J_N'$ in the framing $\tau_N$.
\begin{lem}\label{lem J_N' in tau_N}
We have $\tau_N^*J_N'=J_D$ on vectors tangential to $D$ and
\begin{align*}
(\tau_N^*J_N')\partial_{z_1}&=\partial_{z_2}+r(\partial_{z_1}C)\partial_r+(\partial_{z_2}C)\partial_\theta\\
(\tau_N^*J_N')\partial_{z_2}&=-\partial_{z_1}+r(\partial_{z_2}C)\partial_r-(\partial_{z_1}C)\partial_\theta.
\end{align*}
\end{lem}
\begin{proof}
We have $\tau_N^*J_N'=\chi^*(J_D+J_\Sigma)$. Moreover, $\chi_*$ is the identity on vectors tangential to $D$ and
\begin{align*}
\chi_*\partial_{z_1}&=-(\partial_{z_1}C)\partial_\theta+\partial_{z_1}\\
\chi_*\partial_{z_2}&=-(\partial_{z_2}C)\partial_\theta+\partial_{z_2}.
\end{align*}
This implies the claim.
\end{proof}
We can now prove:
\begin{prop}\label{prop diff K_N' and K_N}
For a surface $L^N$ in $N'$ as above we have
\begin{equation*}
K_N'L^N=K_NL^N+\partial L^N\cdot\left(\sum_{i=1}^{2g}a_i\Gamma^N_i\right).
\end{equation*}
The intersection on the right is taken in the boundary of the tubular neighbourhood.
\end{prop}
\begin{proof}
Let $\nu\Sigma_N$ denote the tubular neighbourhood of radius $2$ and $\nu\Sigma_N'$ the tubular neighbourhood of radius $1$. We can assume that the surface $L^N$ bounds a curve $\partial L^N$ on the boundary of $\nu\Sigma_N'$ and is parallel to this curve on $\nu\Sigma_N\setminus \nu\Sigma_N'$ in the framing $\tau_N$. On the boundary of $\nu\Sigma_N$ we consider the form $\Omega_N$ while on the boundary of $\nu\Sigma_N'$ we have the form $\Omega_N'$. We have calculated the form $\Omega_N'$ in the trivialization $\tau_N$ in Lemma \ref{lem Omega_N' in tau_N} and the almost complex structure in Lemma \ref{lem J_N' in tau_N}. Using a smooth cut-off function $\rho$ in front of the partial derivatives of $C$ in $\tau_N^*J_N'$ and the same cut-off function in front of the $dC$-term in $\tau_N^*\Omega_N'$  we can deform the almost complex structure through almost complex structures compatible with the symplectic form and the 2-form through holomorphic 2-forms without introducing new zeroes such that at the boundary of the tubular neighbourhood at radius $\frac{3}{2}$ the form is equal to 
\begin{equation*}
\Omega_D\wedge\Omega_\Sigma e^{-iC}
\end{equation*}
while the almost complex structure is the standard one. If we want to deform the 2-form to the standard form $\Omega_D\wedge\Omega_\Sigma$ at radius $2$ we have to introduce a zero set Poincar\'e dual on $\Sigma\times S^1$ to the cohomology class corresponding to the $S^1$-valued function $e^{iC}$. This zero set represents the class $\sum_{i=1}^{2g}a_i\Gamma_i^N$ and has with the surface $L^N$ intersection
\begin{equation*}
\partial L^N\cdot\left(\sum_{i=1}^{2g}a_i\Gamma^N_i\right).
\end{equation*}
This implies the claim.
\end{proof}

We can now calculate the numbers $K_XS_i$. We have fixed a normal form basis for $H^2(X)$, in particular a basis of rim tori and dual vanishing surfaces. The basis of rim tori is determined by a basis $\alpha_1,\ldots,\alpha_d$ for $\text{ker}(i_M\oplus i_N)$. We fix such a basis and choose the induced basis of rim tori in all fibre sums that we get by the different choices of gluing diffeomorphism. Since $K_X$ evaluates to zero on rim tori, we can assume by Lemma \ref{lem S1,...,Sd unique} that the surfaces $S_i$ are chosen as in Section \ref{subsection defn vanishing surfaces Si}. We choose the basis for $H_1(\Sigma)$ as in Remark \ref{rem on construction of Ri Si}. In particular, the basis element $\gamma_i^M$ is given by the curve $\alpha_i^M$ for $1\leq i\leq d$.
\begin{lem}\label{lem calc of KXSi} With the choice of orientation as in Definition \ref{defn orientation of surfaces bounding M'}, we have $K_XS_i=K_MS^M_i-K_N'S^N_i$, hence 
\begin{equation*}
K_XS_i=K_MS^M_i-K_NS^N_i-a_i.
\end{equation*}
\end{lem}
\begin{proof} The proof is similar to the proof for Lemma \ref{lem comp K_XB_X}. The minus sign in front of $K_N'S^N_i$ comes in because we have to change the orientation on $S_i^N$ if we want to sew it to $S_i^M$ to get the surface $S_i$ in $X$. The second claim follows from Proposition \ref{prop diff K_N' and K_N}.
\end{proof}
This term can be evaluated more explicitly, because we have
\begin{align*}
S_i^M&=D_i^M\\
S_i^N&=Q_i^N\cup \bar{U}_i^N\cup \bar{D}_i^N,
\end{align*}
where $\bar{U}_i^N$ is constructed from a surface $V_i^N$ in the closed manifold $N$ representing $a_i(B_N-B_N^2\Sigma_N)$ by deleting the part in the interior of $\nu\bar{\Sigma}_N$. There are additional rim tori terms in the definition of the $S_i$ used to separate $S_i$ and $S_j$ for $i\neq j$ which we can ignore here because the canonical class evaluates to zero on them. We think of the surface $Q_i^N$ as being constructed in the region between radius $2$ and $3$ times $\Sigma_N$. We extend the almost complex structure and the holomorphic 2-form over this region without change. By Lemma \ref{lem QiN disjoint from annulus} we can assume that $Q_i^N$ is disjoint from the zero set of $\Omega_D\wedge \Omega_\Sigma$, consisting of $2g-2$ parallel annuli. Hence there are no zeroes of $\Omega_N$ on $Q_i^N$. The surface $\bar{D}_i^N$ contributes $K_ND_i^N$ to the number $K_NS_i^N$ and the surface $\bar{U}_i^N$ contributes
\begin{align*}
K_N\bar{U}_i^N&=a_iK_N(B_N-(B_N^2)\Sigma_N)\\
&= a_i(K_NB_N-(2g-2)B_N^2).
\end{align*}
Hence we get:
\begin{lem}\label{lem KXSi expanded} With our choice of the surfaces $S_i^M$ and $S_i^N$, we have 
\begin{equation*}
K_XS_i=K_MD_i^M-K_ND_i^N-a_i(K_NB_N+1-(2g-2)B_N^2).
\end{equation*}
\end{lem}
This formula has the advantage that the first two terms are independent of the choice of the diffeomorphism $\phi$. 
\begin{defn}
We denote by $X_0$ the manifold obtained from the trivial gluing diffeomorphism that identifies the push-offs in the framings $\tau_M$ and $\tau_N$, characterized by $C=0$.
\end{defn}
By Lemma \ref{lem KXSi expanded} we have
\begin{equation*}
K_{X_0}S_i=K_MD_i^M-K_ND_i^N.
\end{equation*}
Collecting our calculations it follows that we can write
\begin{equation}\label{formula canonical class with sigma term}
K_X=\overline{K_M}+\overline{K_N}+\sum_{i=1}^dr_iR_i+b_XB_X+\sigma_X\Sigma_X,
\end{equation}
where
\begin{align*}
\overline{K_M}&=K_M-(2g-2)B_M-(K_MB_M-(2g-2)B_M^2)\Sigma_M \in P(M)\\
\overline{K_N}&=K_N-(2g-2)B_N-(K_NB_N-(2g-2)B_N^2)\Sigma_N \in P(N)\\
r_i&=K_XS_i=K_{X_0}S_i-a_i(K_NB_N+1-(2g-2)B_N^2)\\
b_X&=2g-2\\
\sigma_X&=K_MB_M+K_NB_N+2-(2g-2)(B_M^2+B_N^2).
\end{align*}
In this formula $K_X$ depends on the diffeomorphism $\phi$ through the term 
\begin{equation*}
-a_i(K_NB_N+1-(2g-2)B_N^2)
\end{equation*}
which gives the contribution
\begin{equation*}
(K_NB_N+1-(2g-2)B_N^2)R_C=-\sum_{i=1}^da_i(K_NB_N+1-(2g-2)B_N^2)R_i
\end{equation*}
to the canonical class. The formula for $K_X$ can be written more symmetrically by also using the class $\Sigma_X'=R_C+\Sigma_X$, induced from the push-off on the $N$ side. We then get: 

\begin{thm}\label{thm on the canonical class Gompf}\index{Canonical class} Let $X=M\#_{\Sigma_M=\Sigma_N}N$ be a symplectic generalized fibre sum of closed oriented symplectic 4-manifolds $M$ and $N$ along embedded symplectic surfaces $\Sigma_M,\Sigma_N$ of genus $g$ and self-intersection zero which represent indivisible homology classes and are oriented by the symplectic forms. Suppose that the cohomology of $M$, $N$ and $X$ is torsion free. Choose a normal form basis for $H^2(X;\mathbb{Z})$ as in Theorem \ref{formula H^2 no torsion}. Then the canonical class of $X$ is given by
\begin{equation*}
K_X=\overline{K_M}+\overline{K_N}+\sum_{i=1}^dt_iR_i+b_XB_X+\eta_X\Sigma_X+\eta_X'\Sigma_X',
\end{equation*}
where
\begin{align*}
\overline{K_M}&=K_M-(2g-2)B_M-(K_MB_M-(2g-2)B_M^2)\Sigma_M \in P(M)\\
\overline{K_N}&=K_N-(2g-2)B_N-(K_NB_N-(2g-2)B_N^2)\Sigma_N \in P(N)\\
t_i&=K_{X_0}S_i\\
b_X&=2g-2\\
\eta_X&=K_MB_M+1-(2g-2)B_M^2\\
\eta_X'&= K_NB_N+1-(2g-2)B_N^2.
\end{align*}
In evaluating $K_{X_0}S_i$ we choose the basis of rim tori in $X_0$ determined by the basis in $X$ and a corresponding dual basis of vanishing surfaces.
\end{thm}

Under the embeddings of $H^2(M)$ and $H^2(N)$ into $H^2(X)$ given by Corollary \ref{general case embedding HM HN in HX}, the canonical classes of $M$ and $N$ map to
\begin{align*}
K_M&\mapsto \overline{K_M}+(2g-2)B_X+(K_MB_M-(2g-2)B_M^2)\Sigma_X\\
K_N&\mapsto \overline{K_N}+(2g-2)B_X+(K_NB_N-(2g-2)B_N^2)\Sigma_X'.
\end{align*}
This implies:
\begin{cor}\label{formula KX embedd HM HN HX} Under the assumptions in Theorem \ref{thm on the canonical class Gompf} and the embeddings of $H^2(M)$ and $H^2(N)$ into $H^2(X)$ given by Corollary \ref{general case embedding HM HN in HX}, the canonical class of the symplectic generalized fibre sum $X=M\#_{\Sigma_M=\Sigma_N}N$ is given by
\begin{equation*}
K_X=K_M+K_N+\Sigma_X+\Sigma_X'-(2g-2)B_X+\sum_{i=1}^dt_iR_i,
\end{equation*}
where $t_i= K_{X_0}S_i$.
\end{cor}
For example, suppose that the genus $g$ is equal to one and there are no rim tori in $X$. Then we get the classical formula for the symplectic generalized fibre sum along embedded tori
\begin{equation*}
K_X=K_M+K_N+2\Sigma_X,
\end{equation*}
which can be found in the literature, e.g.~\cite{Smi}.

\section{Checks and comparison with previously known formulae}\label{gompf can class applic examp}

To check the formula for the canonical class given by Theorem \ref{thm on the canonical class Gompf}, we can calculate the square $K_X^2=Q_X(K_X,K_X)$ and compare it with the classical formula
\begin{equation}\label{formula for c_1^2 Gompf}
c_1(X)^2=c_1(M)^2+c_1^2(N)+8g-8,
\end{equation}
that can be derived independently using the formulae for the Euler characteristic and the signature of a generalized fibre sum (see the proof of Corollary \ref{betti X}) 
\begin{align*}
e(X)&=e(M)+e(N)+4g-4\\
\sigma(X)&=\sigma(M)+\sigma(N)
\end{align*}
and the formula $c_1^2=2e+3\sigma$. We do this step by step. By equation \eqref{formula canonical class with sigma term} we have: 
\begin{align*}
Q_X(\overline{K_M}, \overline{K_M})&= Q_M(\overline{K_M},\overline{K_M})\\
&= Q_M(\overline{K_M},K_M)\\
&= K_M^2-(2g-2)K_MB_M-(2g-2)(K_MB_M-(2g-2)B_M^2)\\
&=K_M^2-(4g-4)K_MB_M+(2g-2)^2B_M^2.
\end{align*}
The second step in this calculation follows since by definition $\overline{K_M}$ is orthogonal to $B_M$ and $\Sigma_M$. Similarly
\begin{equation*}
Q_X(\overline{K_N}, \overline{K_N})=K_N^2-(4g-4)K_NB_N+(2g-2)^2B_N^2. 
\end{equation*}
The rim torus term $\sum_{i=1}^dr_iR_i$ has zero intersection with itself and all other terms in $K_X$. We have
\begin{equation*}
Q_X(b_XB_X,b_XB_X)=(2g-2)^2(B_M^2+B_N^2),
\end{equation*}
and
\begin{equation*}
2Q_X(b_XB_X,\sigma_X\Sigma_X)=2(2g-2)(K_MB_M+K_NB_N+2-(2g-2)(B_M^2+B_N^2)).
\end{equation*}
The self-intersection of $\Sigma_X$ is zero. Adding these terms together, we get the expected result
\begin{equation*}
K_X^2= K_M^2+K_N^2+8g-8.
\end{equation*}

As another check we compare the formula for $K_X$ in Theorem \ref{thm on the canonical class Gompf} with a formula of Ionel and Parker \cite[Lemma 2.4]{IP2} that determines the intersection of $K_X$ with certain homology classes for symplectic generalized fibre sums in arbitrary dimension and without the assumption of trivial normal bundles of $\Sigma_M$ and $\Sigma_N$. For dimension $4$ with surfaces of genus $g$ and self-intersection zero the formula can be written (in our notation for the cohomology of $X$):
\begin{align*}
K_XC&=K_MC\quad \mbox{for $C\in P(M)$}\\
K_XC&=K_NC\quad \mbox{for $C\in P(N)$}\\
K_X\Sigma_X&=K_M\Sigma_M=K_N\Sigma_N\\
&= 2g-2\quad\mbox{(by the adjunction formula)}\\
K_XR&=0\quad \mbox{for all elements in $R(X)$}\\
K_XB_X&=K_MB_M+K_NB_N+2(B_M\Sigma_M=B_N\Sigma_N)\\
&= K_MB_M+K_NB_N+2.
\end{align*}
There is no statement about the intersection with classes in $S'(X)$ that have a non-zero component in $\mbox{ker}\,(i_M\oplus i_N)$. We calculate the corresponding intersections with the formula for $K_X$ in equation \eqref{formula canonical class with sigma term}. For $C\in P(M)$ we have
\begin{align*}
K_X C&=\overline{K_M} C\\
&= K_MC,
\end{align*}
where the second line follows because the terms in the formula for $\overline{K_M}$ involving $B_M$ and $\Sigma_M$ have zero intersection with $C$, being a perpendicular element. A similar equation holds for $N$. The intersection with $\Sigma_X$ is given by
\begin{align*}
K_X\Sigma_X&=(2g-2)B_X\Sigma_X\\
&= 2g-2.
\end{align*}
The intersection with rim tori is zero and 
\begin{align*}
K_X B_X&= b_XB_X^2+\sigma_X\\
&= (2g-2)(B_M^2+B_N^2)+K_MB_M+K_NB_N+2-(2g-2)(B_M^2+B_N^2)\\
&= K_MB_M+K_NB_N+2,
\end{align*}
which also follows by Lemma \ref{lem comp K_XB_X}. Hence with the formula in Theorem \ref{thm on the canonical class Gompf} we get the same result as with the formula of Ionel and Parker. The only difference is that their formula does not determine the rim tori contribution to the canonical class. 

A final check is the following: Consider the formula for the canonical class as in equation \eqref{formula canonical class with sigma term}. Suppose we change the framing of the surfaces $\Sigma_M$ and $\Sigma_N$ but keep the same abstract gluing diffeomorphism between the boundaries. Then we get the same manifold $X$ but the class $\Sigma_X$ and the numbers $r_i$ change. However, the canonical class should stay the same.
\begin{prop}
The formula for the canonical class depends only on the gluing diffeomorphism and not on the framings of the surfaces $\Sigma_M$ and $\Sigma_N$. 
\end{prop}
\begin{proof}
A change of framing is described by a new choice of basis curves on the boundary of the tubular neighbourhoods:
\begin{align*}
\widetilde{\gamma_i^M}&=\gamma_i^M+b_i\sigma^M\\
\widetilde{\gamma_i^N}&=\gamma_i^N+c_i\sigma^N.
\end{align*}
The meridians $\sigma^M$ and $\sigma^N$ stay the same. In the old basis the gluing diffeomorphism $\phi$ is described by
\begin{equation*}
\phi_*\gamma_i^M=\gamma_i^N+a_i\sigma^N.
\end{equation*}
In the new basis we get
\begin{equation*}
\phi_*\widetilde{\gamma_i^M}=\widetilde{\gamma_i^N}+\widetilde{a_i}\sigma^N
\end{equation*}
where $\widetilde{a_i}=a_i-b_i-c_i$. We also get
\begin{align*}
\widetilde{\gamma_i^M}^*&={\gamma_i^M}^*\\
\widetilde{\sigma^M}^*&={\sigma^M}^*-\sum_{i=1}^{2g}b_i{\gamma_i^M}^*.
\end{align*}
Hence we have for the Poincar\'e duals
\begin{align*}
\widetilde{\Gamma_i^M}&=\Gamma_i^M\\
\widetilde{\Sigma^M}&=\Sigma^M-\sum_{i=1}^{2g}b_i\Gamma_i^M
\end{align*}
and therefore 
 \begin{equation*}
 \widetilde{\Sigma}_X=\Sigma_X-\sum_{i=1}^db_iR_i.
 \end{equation*}
The vanishing classes also have to change because they must be orthogonal to $B_X$ and $\widetilde{\Sigma}_X$. We get
\begin{equation*}
\widetilde{S_i}=S_i+b_iB_X-(b_iB_X^2)\Sigma_X,
\end{equation*}
up to an irrelevant rim torus. The manifold $X_0$ also changes to a manifold $\widetilde{X_0}$: The trivial gluing diffeomorphism is now described by $\widetilde{a_i}=0$, hence in the old basis $a_i=b_i+c_i$. The coefficients $r_i$ change to
\begin{equation*}
\widetilde{r_i}=K_{\widetilde{X_0}}\widetilde{S_i}-\widetilde{a_i}(K_NB_N+1-(2g-2)B_N^2).
\end{equation*}
We have to show that
\begin{equation*}
\sum_{i=1}^d\widetilde{r_i}R_i+\sigma_X\widetilde{\Sigma}_X=\sum_{i=1}^dr_iR_i+\sigma_X\Sigma_X.
\end{equation*}
We first calculate $K_{\widetilde{X_0}}\widetilde{S_i}$:
\begin{align*}
K_{\widetilde{X_0}}\widetilde{S_i}&=\left(\sum_{i=1}^d(K_{X_0}S_i-(b_i+c_i)(K_NB_N+1-(2g-2)B_N^2)  )R_i+b_XB_X+\sigma_X\Sigma_X\right)\widetilde{S_i}\\
&=K_{X_0}S_i-(b_i+c_i)(K_NB_N+1-(2g-2)B_N^2)+b_ib_XB_X^2-(b_iB_X^2)b_X+\sigma_Xb_i\\ 
&=K_{X_0}S_i-(b_i+c_i)h+\sigma_Xb_i,
\end{align*}
where $h=(K_NB_N+1-(2g-2)B_N^2)$. Hence we get
\begin{align*}
\widetilde{r_i}&=K_{X_0}S_i-(b_i+c_i)h+\sigma_Xb_i-(a_i-b_i-c_i)h\\
&=K_{X_0}S_i-a_ih+\sigma_Xb_i.
\end{align*}
Therefore
\begin{align*}
\sum_{i=1}^d\widetilde{r_i}R_i+\sigma_X\widetilde{\Sigma}_X&=\sum_{i=1}^d(K_{X_0}S_i-a_ih+\sigma_Xb_i)R_i+\sigma_X\Sigma_X-\sum_{i=1}^d\sigma_Xb_iR_i\\
&=\sum_{i=1}^d(K_{X_0}S_i-a_ih)R_i+\sigma_X\Sigma_X\\
&=\sum_{i=1}^dr_iR_i+\sigma_X\Sigma_X.
\end{align*}
\end{proof}

\section{Generalized fibre sums along tori in cusp neighbourhoods}\label{subsect symp gen fibre sum E(n)}

We consider a special case of Theorem \ref{thm on the canonical class Gompf}, where the generalized fibre sum is along embedded tori in cusp neighbourhoods. Let $M$ and $N$ be closed symplectic 4-manifolds which contain symplectically embedded tori $T_M$ and $T_N$ of self-intersection zero, representing indivisible classes. Suppose that $M$ and $N$ have torsion free homology and both tori are contained in cusp neighbourhoods. Then each torus has two vanishing cycles coming from the cusp. We choose identifications of both $T_M$ and $T_N$ with $T^2=S^1\times S^1$ such that the vanishing cycles are given by the simple closed loops $\gamma_1=S^1\times 1$ and $\gamma_2=1\times S^1$. The loops bound embedded vanishing disks in $M$ and $N$, denoted by $(D_1^M, D_2^M)$ and $(D_1^N, D_2^N)$. The existence of the vanishing disks shows that the embeddings $T_M\rightarrow M$ and $T_N\rightarrow N$ induce the zero map on the fundamental group. 

We choose for both tori trivializations of the normal bundles and corresponding push-offs $T^M$ and $T^N$. By choosing the trivializations appropriately we can assume that the vanishing disks bound the vanishing cycles on these push-offs and are contained in $M\setminus \mbox{int}\,\nu T_M$ and $N\setminus \mbox{int}\,\nu T_N$. The vanishing disks have self-intersection $-1$ if the vanishing cycles on the boundary of the tubular neighbourhood are framed by the normal framing on the push-off. We consider the symplectic generalized fibre sum $X=X(\phi)=M\#_{T_M=T_N}N$ for a gluing diffeomorphism 
\begin{equation*}
\phi\colon \partial(M\setminus \mbox{int}\,\nu T_M)\rightarrow \partial(N\setminus \mbox{int}\,\nu T_N).
\end{equation*}
The vanishing cycles on both tori determine a basis for $H_1(T^2)$. If $a_i=\langle C,\gamma_i\rangle$ and $\sigma$ denotes the meridians to $T_M$ in $M$ and $T_N$ in $N$, then the gluing diffeomorphism $\phi\colon \partial\nu T_M\rightarrow \partial \nu T_N$ maps in homology
\begin{align*}
\gamma_1&\mapsto \gamma_1+a_1\sigma\\
\gamma_2&\mapsto\gamma_2+a_2\sigma\\
\sigma&\mapsto-\sigma
\end{align*}
by Lemma \ref{act phi hom 1}. By Proposition \ref{prop [C] determines phi} the diffeomorphism $\phi$ is determined by the integers $a_1,a_2$ up to isotopy. Note that $H_1(X(\phi))\cong H_1(M)\oplus H_1(N)$ by Theorem \ref{H 1 for X}. Hence under our assumptions the homology of $X(\phi)$ is torsion free. The group of rim tori is $R(X)=\mbox{coker}(i_M^*+i_N^*)\cong \mathbb{Z}^2$. Let $\gamma_1^*,\gamma_2^*$ denote the dual basis of $H^1(T^2)$ and $R_1, R_2$ the associated rim tori in $X$.  

We can calculate the canonical class of $X=X(\phi)$ by Theorem \ref{thm on the canonical class Gompf}: Let $B_M$ and $B_N$ denote surfaces in $M$ and $N$ which intersect $T_M$ and $T_N$ transversely once and are disjoint from the vanishing disks. Then the canonical class is given by
\begin{equation*}
K_X=\overline{K_M}+\overline{K_N}+(t_1R_1+t_2R_2)+b_XB_X+\eta_X T_X+\eta_X' T_X',
\end{equation*}
where
\begin{align*}
\overline{K_M}&=K_M-(K_MB_M)T_M \in P(M)\\
\overline{K_N}&=K_N-(K_NB_N)T_N \in P(N)\\
t_i&=K_{X_0}S_i\\
b_X&=2g-2=0\\
\eta_X&=K_MB_M+1\\
\eta_X'&=K_NB_N+1.
\end{align*}
Here $T_X$ is the torus in $X$ determined by the push-off $T^M$ and $T_X'$ is determined by the push-off $T^N$.
\begin{lem}\label{SW basic class argum on -2 spheres} In the situation above we have $K_{X_0}S_i=0$ for $i=1,2$.
\end{lem}
\begin{proof} The pairs $(D_1^M, D_1^N)$ and $(D_2^M, D_2^N)$ sew together in the generalized fibre sum $X_0=X(\phi_0)$, where $\phi_0$ denotes the trivial gluing diffeomorphism that identfies the push-offs, and determine embedded spheres $S_1, S_2$ of self-intersection $-2$. We claim that  
\begin{equation*}
K_{X_0}S_i=0, \quad i=1,2.
\end{equation*}
This is clear by the adjunction formula if the spheres are symplectic or Lagrangian. In the general case, there exist rim tori $R_1, R_2$ in $X_0$ which are dual to the spheres $S_1, S_2$ and which can be assumed Lagrangian by the Gompf construction. Consider the pair $R_1$ and $S_1$: By the adjunction formula we have $K_{X_0}R_1=0$. The sphere $S_1$ and the torus $R_1$ intersect once. By smoothing the intersection point we get a smooth torus of self-intersection zero in $X_0$ representing $R_1+S_1$. In our case, we have $b_2^+(X_0)\geq 3$, hence the canonical class $K_{X_0}$ is a Seiberg-Witten basic class. The adjunction inequality \cite[Theorem 2.4.8]{GS} implies that $K_{X_0}(R_1+S_1)=0$, which shows that $K_{X_0}S_1=0$. In a similar way it follows that $K_{X_0}S_2=0$.
\end{proof}

This implies:
\begin{prop}\label{prop cusp tori canonical class sum formula} Let $M$ and $N$ be closed symplectic 4-manifolds with torsion free homology. Suppose that $T_M$ and $T_N$ are embedded symplectic tori of self-intersection zero which are contained in cusp neighbourhoods in $M$ and $N$ and represent indivisible classes. Then the canonical class of the symplectic generalized fibre sum $X=X(\phi)=M\#_{T_M=T_N}N$ is given by
\begin{align*}
K_X&=\overline{K_M}+\overline{K_N}+\eta_X T_X+\eta'_XT_X'\\
&=K_M+K_N+T_X+T_X',
\end{align*}
where
\begin{align*}
\overline{K_M}&=K_M-(K_MB_M)T_M \in P(M)\\
\overline{K_N}&=K_N-(K_NB_N)T_N \in P(N)\\
\eta_X&=K_MB_M+1\\
\eta_X'&=K_NB_N+1.
\end{align*}
The second line in the formula for $K_X$ holds by Corollary \ref{formula KX embedd HM HN HX} under the embeddings of $H^2(M)$ and $H^2(N)$ in $H^2(X)$.
\end{prop}
As a special case, suppose that the tori $T_M$ and $T_N$ are contained in smoothly embedded nuclei $N(m)\subset M$ and $N(n)\subset N$, which are by definition diffeomorphic to neighbourhoods of a cusp fibre and a section in the elliptic surfaces $E(m)$ and $E(n)$, cf.~\cite{Gnuc, GS}. The surfaces $B_M$ and $B_N$ can then be chosen as the spheres $S_M, S_N$ inside the nuclei corresponding to the sections. The spheres have self-intersection $-m$ and $-n$ respectively. If the sphere $S_M$ is symplectic or Lagrangian in $M$, we get by the adjunction formula
\begin{equation*}
K_MS_M=m-2.
\end{equation*}
If $m=2$ this holds by an argument similar to the one in Lemma \ref{SW basic class argum on -2 spheres} already without the assumption that $S_M$ is symplectic or Lagrangian. With Proposition \ref{prop cusp tori canonical class sum formula} we get:
\begin{cor}\label{cor tori in nuclei can class sum} Let $M$ and $N$ be closed symplectic 4-manifolds with torsion free homology. Suppose that $T_M$ and $T_N$ are embedded symplectic tori of self-intersection zero which are contained in embedded nuclei $N(m)\subset M$ and $N(n)\subset N$. Suppose that $m=2$ or the sphere $S_M$ is symplectic or Lagrangian. Similarly, suppose that $n=2$ or the sphere $S_N$ is symplectic or Lagrangian. Then the canonical class of the symplectic generalized fibre sum $X=X(\phi)=M\#_{T_M=T_N}N$ is given by
\begin{equation*}
K_X=\overline{K_M}+\overline{K_N}+(m-1)T_X+(n-1)T_X',
\end{equation*}
where
\begin{align*}
\overline{K_M}&=K_M-(m-2)T_M \in P(M)\\
\overline{K_N}&=K_N-(n-2)T_N \in P(N).\\
\end{align*}
\end{cor}

As an example we consider {\em twisted fibre sums} of elliptic surfaces $E(m)$ and $E(n)$ which are glued together by diffeomorphisms that still preserve the $S^1$-fibration on the boundary of the tubular neighbourhood, but not the $T^2$-fibration as in the standard fibre sum (compare with \cite{Pa}):

\begin{ex}\label{general fibre sum elliptic twist} Suppose that $M=E(m)$ and $N=E(n)$ with general fibres $T_M$ and $T_N$. The framing for the tori is given by the framing induced from the elliptic fibration. Since $K_{E(m)}=(m-2)T_M$ and the spheres in the nuclei are symplectic, the canonical class of $X=X(\phi)=E(m)\#_{T_M=T_N}E(n)$ is given by
\begin{equation*}
K_X=(m-1)T_X+(n-1)T_X'.
\end{equation*}
Using the identity $T_X'=T_X+R_C$ this formula can be written as
\begin{equation*}
K_X=(m+n-2)T_X+(n-1)R_C.
\end{equation*}
Note that $R_C=-(a_1R_1+a_2R_2)$. If both coefficients $a_1$ and $a_2$ vanish and hence the gluing diffeomorphism is isotopic to the trivial diffeomorphism, we get the standard formula 
\begin{equation*}
K_X=(m+n-2)T_X
\end{equation*}
for the fibre sum $E(m+n)=E(m)\#_{T_M=T_N}E(n)$. In general, if $n=1$ it follows that there is no rim tori contribution to the canonical class, independent of the gluing diffeomorphism $\phi$. This can be explained as follows: Every orientation preserving self-diffeomorphism of $\partial (E(1)\setminus \mbox{int}\,\nu T)$ extends over $E(1)\setminus \mbox{int}\,\nu T$, where $T$ denotes a general fibre \cite[Theorem 8.3.11]{GS}. Hence all generalized fibre sums $X(\phi)$ are diffeomorphic to the elliptic surface $E(m+1)$ in this case. The same argument holds if $n\neq 1$ but $m=1$. Using the Seiberg-Witten invariants of elliptic surfaces we see that the canonical class is always the standard one.

If both $m$ and $n$ are different from $1$, there may exist a non-trivial rim tori contribution. For example, if we consider the generalized fibre sum $X=X(\phi)=E(2)\#_{T_M=T_N}E(2)$ of two $K3$ surfaces $E(2)$, then
\begin{align*}
K_X&=2T_X-(a_1R_1+a_2R_2)\\
&=T_X+T_X'
\end{align*}
If the greatest common divisor of $a_1$ and $a_2$ is odd, then $K_X$ is indivisible (because there exist certain vanishing classes in $X$ dual to the rim tori $R_1$ and $R_2$). In this case the manifold $X$ is no longer spin, hence cannot be homeomorphic to the spin manifold $E(4)$. 
\end{ex}

\bibliographystyle{amsplain}

\bigskip
\bigskip

\end{document}